\documentclass[final,onefignum,onetabnum]{siamonline171218}
\usepackage{lipsum}
\usepackage{amsfonts}
\usepackage{graphicx}
\usepackage{lineno}
\usepackage{epstopdf}
\usepackage[multiple]{footmisc}
\usepackage{algorithmic}
\ifpdf
  \DeclareGraphicsExtensions{.eps,.pdf,.png,.jpg}
\else
  \DeclareGraphicsExtensions{.eps}
\fi

\usepackage{enumitem}
\setlist[enumerate]{leftmargin=.5in}
\setlist[itemize]{leftmargin=.5in}

\newsiamremark{remark}{Remark}
\newsiamremark{hypothesis}{Hypothesis}
\crefname{hypothesis}{Hypothesis}{Hypotheses}
\newsiamthm{claim}{Claim}
\usepackage{amsopn}

\usepackage[margin=1.0in]{geometry}
\usepackage[utf8]{inputenc}
\usepackage{amssymb}
\usepackage{amsmath}
\usepackage{bbm}
\usepackage{amsfonts}
\usepackage{csquotes}
\MakeOuterQuote{"}
\usepackage{graphicx}
\usepackage{framed} 
\usepackage{epstopdf}
\usepackage{mathrsfs}
\usepackage[multiple]{footmisc}
\usepackage{setspace}

\newcommand{\revdone}{}
\usepackage{listings}
\usepackage{color} 
\definecolor{mygreen}{RGB}{28,172,0} 
\definecolor{mylilas}{RGB}{170,55,241}

\title{Density Estimation in Uncertainty Propagation Problems Using a Surrogate Model\thanks{
\funding{The research of A.\ Ditkowski was supported by the United States - Israel Binational Science Foundation under grant number 2016197. The research of G.\ Fibich and A.\ Sagiv was partially supported by the Israel Science
Foundation (ISF) under grant number 177/13.}}}

\author{Adi Ditkowski\thanks{School of Mathematical Sciences, Tel Aviv University, Tel Aviv 6997801, Israel}
\and Gadi Fibich\footnotemark[2]
\and Amir Sagiv\footnotemark[2] \thanks{\email{asagiv88@gmail.com}}}

\begin{document}

\maketitle

\begin{abstract}
The effect of uncertainties and noise on a quantity of interest (model output) is often better described by its probability density function (PDF) than by its moments. Although density estimation is a common task, the adequacy of approximation methods {\revdone (surrogate models)} for density estimation has not been analyzed before in the uncertainty-quantification (UQ) literature. In this paper, we first show that standard surrogate models (such as generalized polynomial chaos), which are highly accurate for moment estimation, might completely fail to approximate the PDF, even for one-dimensional noise. {\revdone This is because density estimation requires that the surrogate model accurately approximates the gradient of the quantity of interest, and not just the quantity of interest itself. Hence, we develop a novel spline-based algorithm for density-estimation whose convergence rate in $L^q$ is polynomial} in the sampling resolution. This convergence rate is better than that of standard statistical density-estimation methods (such as histograms and kernel density estimators) at dimensions $1 \leq d\leq \frac{5}{2}m$, where $m$ is the spline order. {\revdone Furthermore, we obtain the convergence rate for density estimation with any surrogate model that approximates the quantity of interest and its gradient in $L^{\infty}$.} Finally, we demonstrate our algorithm for problems in nonlinear optics and fluid dynamics. 
\end{abstract}



\section{Introduction}
Uncertainties and noise are prevalent in mathematical models in all branches of science. In such cases, the solution of the (otherwise deterministic) model becomes random, and so one is interested in computing its statistics. This problem, sometimes known as {\em forward uncertainty propagation} (UQ), arises in various areas such as biochemistry \cite{ le2010asynchronous, najm2009uncertainty}, fluid dynamics~\cite{chen2005uncertainty, zabaras2007sparse, le2010spectral,    najm2009uncertainty}, structural engineering \cite{sudret2000stochastic}, hydrology \cite{colombo2018basins}, and nonlinear optics \cite{best2017paper}.

In many applications, one is interested in computing the {\em probability density function} (PDF) of a certain "quantity of interest" (output) of the model \cite{ablowitz2015interacting, chen2005uncertainty, colombo2018basins,  zabaras2007sparse, le2010asynchronous,  best2017paper, ullmann2014pod}. {\revdone Often, density estimation is performed using standard uncertainty propagation methods and surrogate models~\cite{ghanem2017handbook, sudret2000stochastic}, such as Stochastic Finite Element and generalized Polynomial Chaos~(gPC)~
\cite{ghanem2003stochastic, o2013polynomial, stefanou2009stochastic, xiu2010numerical}, hp-gPC~\cite{wan2005adaptive}, and Wiener-Haar expansion \cite{lemaitre2004wavelet}, since these methods can approximate moments with spectral accuracy \cite{xiu2005colocation, xiu2002galerkin}. In this paper we show, however, that methods which are robust and highly accurate for moment-approximation are not necessarily so for density estimation. To the best of our knowledge, this observation has not been made before in the UQ literature. }

{\revdone Why is it then that robust moment approximation does not imply robust density estimation? This is because the quantity of interest $f(\alpha)$ and its PDF $p_f(\alpha)$ are explicitly related by (Lemma~\ref{lem:pdf}) $$p_f (y) = \sum\limits_{\alpha \in f^{-1}(y)} \frac{c(\alpha)}{|f'(\alpha)|} \, ,$$ where $\alpha$ is the one-dimensional random parameter and $c(\alpha)d\alpha$ is its distribution. This formula and its multidimensional counterpart (Lemma \ref{lem:ddim_pdf}) show that even if $f$ is well approximated by a function $g$ in $L^q$, the corresponding density $p_g$ might not be a good approximation of $p_f$. Indeed, for $p_g$ to approximate $p_f$, then $g'$ needs to be close to $f'$, and $g'(\alpha)$ should also vanish if and only if $f'(\alpha)$ does. These conditions might not be satisfied by some of the above-mentioned standard UQ methods. In contrast, spline interpolation approximates both the function and its gradient~\cite{beatson1986spline, hall1976bounds, rice1978spline, schultz1969spline}, and does not tend to produce spurious extremal points. Therefore, we construct a novel algorithm for density estimation in uncertainty-propagation problems using splines as our surrogate model.} With cubic splines, our density-estimation algorithm has a {\em guaranteed} convergence rate of at least~$h^3$, where $h$ is the maximal sampling spacing (resolution). More generally, with splines of order $m$, the convergence rate is at least $h^m$. These rates are superior to those of the standard kernel density estimators~\cite{tsybakov2009estimate, wasserman2013all}, for noise dimension $1\leq d \leq \frac{5}{2}m$. {\revdone Our choice of splines is motivated by the availability and efficiency of one- and multi-dimensional spline toolboxes. Nevertheless, other surrogate models can be used in this algorithm, and indeed this paper lays the theoretical framework for deriving the convergence rate of such methods (Corollaries \ref{cor:1d_gen} and \ref{cor:gen_multid}). We show, essentially, that density estimation convergence can be performed with any surrogate model for which the $L^{\infty}$ error of both the function and its gradient converge to zero as the spacing resolution $h$ vanishes.}
Because we only rely on solving the underlying deterministic model (i.e., our method is non-intrusive), and because interpolation by spline is a standard numerical procedure, our proposed method is very easy to implement.

{\revdone While the focus of this paper is on density-estimation, we also consider the problem of moment-estimation using a small sample-size.} Traditionally, the error bounds of moment-estimation for spectral methods (e.g., gPC) are obtained asymptotically as $N$, the number of samples, goes to infinity. In some applications, however, each solution of the deterministic model is computationally expensive and so the number of samples is limited to e.g., $N<100$. Hence, spectrally-convergent methods might fail to attain the desired accuracy due to insufficient sampling resolution, {\em even for one-dimensional noise}. In contrast, the spline-based method approximates moments accurately even when the sample size is small. In addition, high-derivatives and discontinuities have little effect on our method's accuracy, due to the fact that spline interpolation is predominantly local (see Sec.\ \ref{sec:spline}). Another advantage over gPC is that splines are not limited to a specific choice of sampling points.

The paper is organized as follows. Sec.\ \ref{sec:set} introduces the general settings and notations, and presents several density-estimation applications from the forward uncertainty propagation literature. Sec.\ \ref{sec:methods} reviews standard statistical density-estimation methods (histogram, kernel density estimators) and the gPC method for moment- and density-estimation. In Sec.\ \ref{sec:spline} we present our spline-based algorithm for moment- and density-estimation in the one-dimensional case. We then prove that the density-estimation error scales as $N^{-m}$, where $N$ is the number of samples and~$m$~is the order of the splines (Theorem \ref{thrm:cub_pdf}). Sec.\ \ref{sec:multid_theory} generalizes our algorithm to $d$-dimensional noises using tensor-product splines of order $m$. This section also contains our key theoretical result (Theorem~\ref{thm:multid}), that the density-estimation error in the $d$-dimensional case scales as $N^{-\frac{m}{d}}$ .

In Sec.\ \ref{sec:cnls} we compare numerically the moment-estimation and density-estimation accuracy of our spline-based method with that of gPC and KDE in one dimension. In addition, in Sec.\ \ref{sec:nonsmooth} we show that both gPC and our spline-based method can approximate moments and the PDF of certain non-smooth quantities of interest. We conclude this section with two- and three-dimensional numerical examples (Sec.\ \ref{sec:multid_toy}). In all cases, the density-estimation errors are consistent with our error estimates (Theorems \ref{thrm:cub_pdf} and \ref{thm:multid}). We use our method to compute the PDF of the rotation angle of the polarization ellipse in nonlinear optics~(Sec.~\ref{sec:cnls}),~and the PDF of the shock location in the Burgers equation (Sec.\ \ref{sec:burgers}). In all these cases, we confirm that the spline-based density estimation converges at least at a cubic rate, and observe that the spline-based moments are more accurate than the gPC ones for small sample sizes. Sec.~\ref{sec:final}~concludes with open questions and future research directions.

\section{Settings and computational goals}\label{sec:set}

We consider initial value problems of the form
\begin{equation}\label{eq:general_settings}
u_t (t,{\bf x }; \pmb{\alpha } ) = Q(u,{\bf x};\pmb{\alpha}) u  \, ,\qquad u(t=0,{\bf x};\pmb{\alpha}) = u_0 ({\bf x};\pmb{\alpha}) \, ,
\end{equation} 
where ${\bf x}\in \mathbb{R}^d$, $Q$  is a possibly nonlinear differential operator, and $\pmb{\alpha} \in \Omega \subset \mathbb{R}^m$ is a random variable which is distributed according to a continuous weight function $c(\pmb{\alpha})$, {\revdone the PDF of the input parameters,} such that $\int_{\Omega} c(\pmb{\alpha} ) \, d\pmb{\alpha} = 1$. The randomness of $u(t,{\bf x};\pmb{\alpha})$ is due to the dependence of $Q$ and/or $u_0$ on $\pmb{\alpha}$.

For a given a {\em quantity of interest} $f(\pmb{\alpha}):=f(u(t,{\bf x}));\pmb{\alpha}))$, we may wish to perform: 
\begin{enumerate}
\item \textbf{Moment estimation.} Compute the mean, variance, or standard deviation of $f(\pmb{\alpha})$ :
\begin{equation}\label{eq:moments_def}
\mathbb{E} _{\pmb{\alpha}}\lbrack f \rbrack := \int\limits_{\Omega } f(\pmb{\alpha}) \, c(\pmb{\alpha})d\pmb{\alpha} \, , \quad \mbox{Var} \left[ f \right] := \left|\mathbb{E} _{\pmb{\alpha}} \left[ f \right]\right|^2 - \mathbb{E}_{\pmb{\alpha}} \left[ |f| ^2 \right] \, ,\quad \sigma (f) := \sqrt{\mbox{Var} \left[ f \right]} ~ .
\end{equation}
\item \textbf{Density estimation.} Compute the probability distribution function (PDF) of $f(\pmb{\alpha})$.
\begin{subequations}\label{eq:pdf_de}
\begin{equation}
p(y) := \frac{dP(y)}{dy} \, , \qquad y\in \mathbb{R} \, ,
\end{equation}
where $P$ is the cumulative distribution function (CDF),
\begin{equation}
P(y):= \mbox{Prob}\{  f(\pmb{\alpha}) <y \}  \, .
\end{equation}
\end{subequations} 
\end{enumerate}

\subsection{Applications}

Two examples of density-estimation in UQ which will be discussed in this paper are the effect of shot-to-shot variation in nonlinear optics (Sec.\ \ref{sec:cnls}) and hydrodynamical shock formation (Sec.\ \ref{sec:burgers}). We briefly present two other examples of density estimation in the UQ literature, for which our method can also be applied:
\begin{enumerate}
\item {\bf Out-of-equilibrium chemical reactions.} Belosouv-Zhabotinsky type systems model out-of-equilibrium chemical reactions. One concrete system is the Oregonator~\cite{field1974BZ}
\begin{equation*}
\begin{aligned}
\frac{dX}{dt}&= k_1 Y-k_2XY+k_3 X -k_4 X^2 \, ,\\ \frac{dY}{dt} &= -k_1 Y -k_2 XY + k_5 Z \, , \\ \frac{dZ}{dt} &= k_3 X -k_5 Z \, ,
\end{aligned}
\end{equation*}
where $X$, $Y$, and $Z$ are the concentrations of three different chemical species, and $\left\{ k_i\right\}_{i=1}^{5}$ are the rate-parameters, often estimated empirically \cite{najm2009uncertainty}. For large values of $t$, this system exhibits sustained, temporal oscillations with a frequency $F=F(k_1 , \ldots ,k_5)$. To deal with an uncertainty in the parameters $k_4$ and $k_5$, the authors of \cite{le2010asynchronous} computed the moments of $X,Y,Z$, and the PDF of the oscillations frequency $F$. This is an example of \eqref{eq:general_settings}--\eqref{eq:pdf_de} with $\pmb{\alpha} = (k_4, k_5)$ and $f= X$, $Y$, $Z$ and $F$.

\item {\bf Heat convection.} Consider the flow of a fluid in a two-dimensional box ${\bf x} = (x,y) \in [x_1,x_2]\times[y_1,y_2]$, which is modeled by  the Navier-Stokes like equations
$$ \nabla \cdot {\bf u} = 0 \, ,\qquad \frac{\partial \theta}{\partial t} + {\bf u} \cdot \nabla \theta = \nabla ^2 \theta \, , $$
$$\frac{\partial {\bf u}}{\partial t} + {\bf u}\cdot \nabla {\bf u} = -\nabla p + {\rm Pr} \nabla ^2 {\bf u} + F({\bf u},\theta) \, ,$$
where ${\bf u} (t,{\bf x};\pmb{\alpha})$ is the fluid velocity, $p(t,{\bf x};\pmb{\alpha})$ is the pressure, $\theta(t,{\bf x}; \pmb{\alpha})$ is the temperature, ${\rm Pr}$ is the Prandtl number, and $F$ is the buoyant force \cite{zabaras2007sparse}. The temperature is a known constant $\theta _0$ on one side of the box, but is random on the other side,~i.e., $$\theta(t,x_1,y) \equiv \theta _0 \, , \qquad \theta(t,x_2,y) = \theta_1 (y;\pmb{\alpha}) \, .$$
The PDF of the pressure and of the velocity were computed in \cite{ullmann2014pod} when $\theta_1(y;\alpha) =\theta_1 (\alpha)$ and~$\alpha$ is uniformly distributed in~$[\alpha _{\min}, \alpha_{\max} ]$, and in \cite{zabaras2007sparse} when $\theta_1(y;\pmb{\alpha})$~is a Gaussian random process. 

\end{enumerate}

\section{Review of existing methods}\label{sec:methods}
We briefly present the standard methods in the literature for~\eqref{eq:general_settings}--\eqref{eq:pdf_de}.
\subsection{Monte-Carlo method, the histogram method, and Kernel Density Estimators}\label{sec:kde}
Given $N$ independently and identically distributed (iid) samples $\left\{ \pmb{\alpha} _j \right\}_{j=1}^N$, the simplest moment estimator is the Monte-Carlo approximation $E_{\pmb{\alpha}} \left[ f \right] \approx~\frac{1}{N} \sum_{n=1}^N f(\pmb{\alpha }_n)$. The Monte Carlo method is intuitive and easy to implement. The main drawback of this method is its slow convergence rate of $O(N^{-1/2})$. In cases where each computation of $f(\pmb{\alpha}_j)$~is expensive (e.g., when it requires to solve numerically \eqref{eq:general_settings} with $\pmb{\alpha}=\pmb{\alpha}_j$), this slow convergence rate can make the Monte-Carlo method impractical.

{\em Density estimation} using $N$ iid samples of $f(\pmb{\alpha})$, denoted by $\left\{ f_j \right\}_{j=1}^N$, is a fundamental problem in non-parametric statistics. A widely-used method for density estimation is the histogram method, in which one partitions the range of~$f(\pmb{\alpha} )$ into $L$ disjoint bins $\left\{ B_{\ell} \right\}_{\ell=1}^{L}$, and approximates the PDF $p$ with the {\em histogram estimator}
\begin{equation}\label{eq:hist}
p_{\rm hist}(y) :\,= \frac{1}{N}\sum\limits_{\ell =1}^L  \left( \#~\text{of samples for which}~f_j \in B_{\ell} \right)\cdot \mathbbm{1}_{ B_{\ell} } (y)  \, ,
\end{equation}
where $\mathbbm{1}_{B_{\ell} }$ is the characteristic function of bin $B_{\ell}$ \cite{wasserman2013all}. An alternative approach (which, unlike  the histogram method, can provide a smooth PDF) is the Kernel Density Estimator~(KDE)
\begin{equation}\label{eq:KDE}
p_{\rm kde} (y) :\,= \frac{1}{Nh} \sum\limits_{j=1}^N K\left( \frac{y - f_j}{h} \right) \, ,
\end{equation}
where $h>0$ is the "window size" and $K$ is the kernel function (e.g., $K(t) = (2\pi)^{-1/2}e^{-t^2 /2}$), see~\cite{tsybakov2009estimate, wasserman2013all}. The $L^1$ error of KDE method asymptotically scales as $N^{-2/5}$ \cite{devroye1985l1}.\footnote{The mean $L^2$ error (the squared root of the "MISE"), also asymptotically scales as $N^{-\frac{2}{5}}$ \cite{tsybakov2009estimate, wasserman2013all}.} As with the Monte-Carlo method, this rate is too slow when each evaluation of~$f_j$~is computationally expensive.

\subsection{Generalized Polynomial Chaos}

The Monte-Carlo method, the histogram method, and KDE are all statistical methods, in the sense that they only rely on the sampled values $\left\{f_j \right\}_{j=1}^N$. Much more information can be extracted from $\left\{f_j \right\}_{j=1}^N$ if the two following conditions hold: 
\begin{enumerate}
\item The "original" $\left\{\pmb{\alpha}_j \right\} _{j=1}^N $ for which $f(\pmb{\alpha} _j)=f_j$ are known.
\item $f(\pmb{\alpha})$ is smooth.\footnote{In Sec.\ \ref{sec:nonsmooth} we show how our method can be extended to non-smooth functions.}
\end{enumerate}
These two conditions often hold in the general settings of Sec.\ \ref{sec:set}. In such cases, a powerful numerical approach, known as  generalized Polynomial Chaos (gPC), can be applied \cite{ghanem2017handbook, ghanem2003stochastic, o2013polynomial, xiu2010numerical}. For clarity, we review the gPC method for a one-dimensional random variable $\alpha$,  i.e.,~$\Omega \subseteq~\mathbb{R}$.

We define the set of orthogonal polynomials $\{ p_n (x) \}_{n=0} ^{\infty }$ with respect to $c(\alpha)$ by the conditions~\cite{szego1939orthogonal}
\begin{equation}\label{eq:orthogonal_polynomials_def}
{\rm Deg} ( p_n ) = n \, , \qquad  \int\limits_{\Omega} p_n ^* (\alpha) p_m  (\alpha) \, c (\alpha) d\alpha =  \delta _{n,m} \, ,
\end{equation}
{\revdone where $p_n ^*$ denotes the complex conjugation of $p_m$.} This family of orthogonal polynomials constitutes an orthonormal basis of the space of square integrable functions, i.e., for all $f\in L^2(\Omega, c)$,
\begin{equation}\label{eq:pc_expansion}
f(\alpha ) = \sum\limits_{n=0}^{\infty } \hat{f} (n) p_n (\alpha ) \, , \qquad \hat{f}(n) := \int\limits_{\Omega} f(\alpha) p_n (\alpha)  c(\alpha) \, d\alpha, \quad  n= 0,1,\ldots \, . 
\end{equation}
This expansion {\em converges spectrally} for the classical families of orthogonal polynomials, e.g., the normalized Hermite and Legendre polynomials.\footnote{i.e., if $f$ is in $C^{r}$, then $\{\hat{f} (n)\} \leq cn^{-r}$, and if $f$ is analytic, then $\lvert\hat{f} (n)\rvert \leq ce^{-\gamma n}$, for some $c,\gamma >0$.} Specifically, if $f$ is analytic, the truncated expansion~\eqref{eq:pc_expansion} has the exponential accuracy
\begin{equation}\label{eq:spec_l2}
\left\| f(\alpha ) - \sum\limits_{n=0}^{N-1 } \hat{f} (n) p_n (\alpha ) \right\|_2  \sim C e^{-\gamma N} \, ,\qquad N\gg 1 \, ,
\end{equation}
for some constants $C,\gamma >0$ \cite{trefethen2013book, wang2012convergence, xiu2010numerical}.

The expansion coefficients~$\{\hat{f} (n)\}$, see \eqref{eq:pc_expansion}, can be approximated using the Gauss quadrature formula
$\mathbb{E}_{\alpha} \lbrack  g \rbrack \approx \sum_{j=1}^N g (\alpha _j  ) w_j $,
where~$\left\{ \alpha _j\right\}_{j=1}^{N}$ are the distinct and real roots of $p_N(\alpha)$, $w_j  :=~\int\limits_{\Omega } l_j  (\alpha ) \, d\mu (\alpha )$ are the weights, and $l_j  (\alpha)$ are the Lagrange interpolation polynomials with respect to $\left\{ \alpha _j\right\}_{j=1}^{N}$ \cite{davis1967integration}, yielding
\begin{equation}\label{eq:coefs_with_quad}
\hat{f} (n)  \approx \hat{f}_N (n) :\,= \sum\limits_{j=1}^N f\left( \alpha _j  \right) p_n  \left( \alpha _j  \right) w_j ,\qquad n=0,1,\ldots, N-1 \, .
\end{equation}

The gPC collocation approximation is defined by
\begin{equation}\label{eq:col_gpc}
f_N ^{\rm gpc} (\alpha) :\,= \sum\limits_{n=0}^{N-1} \hat{f} _N (n) p_n (\alpha) \, ,
\end{equation} where $\{\hat{f}_N (n)\}_{n=0}^{N-1}$ are given by \eqref{eq:coefs_with_quad}, see \cite{xiu2005colocation}.

The spectral accuracy of the gPC approximation {\revdone in $L^2$} implies a similar accuracy for the approximation of moments:
\begin{corollary}\label{corr:moment_gpc}
Let $f$ be analytic, and let $f_N ^{\rm gpc}$ be its {\rm gPC} collocation approximation of order~$N$, see \eqref{eq:col_gpc}. Then the moments \eqref{eq:moments_def} of $f$ can be approximated by the respective moments of $f_N ^{\rm gpc}$ with exponential accuracy as $N\to \infty$.
\end{corollary}
\begin{proof}
See Appendix \ref{app:l2bounds_pf}.
\end{proof}
\noindent
For a smooth quantity of interest $f$, this spectral convergence rate is superior to the Monte-Carlo's $1/\sqrt{N}$ convergence rate, which explains the popularity of the gPC collocation method.  

In \cite{best2017paper} we used the gPC approximation for moments and density estimation:

\begin{algorithm}
\caption{gPC-based estimation \cite{best2017paper}}\label{alg:gpc_int}
Let $\left\{\alpha _j, w_j \right\}_{j=1}^N$ be the points and weights of the Gaussian quadrature rule of order $N$ that correspond to the weight function~$c(\alpha)$, and let $\left\{ p_n (\alpha) \right\}_{n=0}^{\infty}$ be the respective orthogonal polynomials.
\begin{algorithmic}[1]
\STATE{For $j=1, \dots ,N$, solve \eqref{eq:general_settings} with $\alpha = \alpha_j$ to obtain  $u \left(t, {\bf x}  ; \alpha _j  \right)$.}
\STATE{Approximate $$u (t, x;\alpha ) \approx u _N^{\rm gpc} (t, {\bf x};\alpha ) \, ,$$where $$u_N^{\rm gpc} (t,{\bf x};\alpha) :=  \sum\limits_{n=0}^{N-1} \hat{u}_N (t, {\bf x};n) p_n(\alpha )$$ and
$$\hat{u}_N (u,{\bf x};n ) = \sum\limits _{j=1} ^N p_n (\alpha _j  ) u \left(t, {\bf x}; \alpha _j \right) w_j \, , \qquad n=0,\ldots , N-1 \, .$$}

\STATE{Approximate $f(\tilde{\alpha} _m) \approx f(u _N ^{\rm gpc}(\cdot, \tilde{\alpha } _m ) )$ on a sample of $M\gg N$ points $\left\{ \tilde{\alpha}_m \right\} _{m=1}^M$ which are i.i.d. according to $c(\alpha)$.}
\IF{{\bf goal is moment estimation:}}
\STATE{Use the trapezoidal integration rule with $\left\{ f(\tilde{\alpha}_{m}) \right\}_{m=1}^{M}$ to approximate $\mathbb{E}_{\alpha} [f]$.}\footnotemark
\ELSIF{{\bf goal is density estimation:}}
\STATE{Use the histogram method \eqref{eq:hist} with $\left\{ f( \tilde{\alpha } _m ) \right\}_{m=1}^M$ to estimate the {\rm PDF} of $f$.}
\ENDIF
\end{algorithmic}
\end{algorithm}
\footnotetext{Any standard integration technique could work here, provided sufficient smoothness. If $f(\alpha)$ is smooth, one can approximate $\mathbb{E}_{\alpha} = \hat{f} (0) \approx \hat{f}_N (0)$, see \cite{xiu2010numerical}.}
\noindent
Because of its spectral accuracy (Corollary \ref{corr:moment_gpc}), the number of sample points that is required for gPC to achieve a certain precision is considerably smaller than for Monte-Carlo. To the best of our knowledge, however, there is no convergence result for density estimation using gPC which is analogous to Corollary \ref{corr:moment_gpc}.

Algorithm \ref{alg:gpc_int} can also approximate {\em non-smooth} quantities of interest $f(\alpha)$, as long as~$u(\cdot ;\alpha)$ is smooth, see Sec.\ \ref{sec:cnls} and \cite{best2017paper}. The choice of the histogram method in step 4 is discussed in Sec.~\ref{sec:final}.

The evaluation of $\left\{ f( u _N^{\rm gpc} (\cdot, \tilde{\alpha } _m ) ) \right\}_{m=1}^M$ in step 3 is computationally cheap, as it amounts to a substitution in a polynomial. Therefore, there is essentially no computational cost for choosing $M$ to be sufficiently high for the histogram method. This algorithm is also non-intrusive, in the sense that it only requires direct simulations of the deterministic system~\eqref{eq:general_settings}~with specific $\alpha_j$ values (as opposed to, e.g., Galerkin-type methods~\cite{debusschere2004numerical, 
lemaitre2004wavelet, xiu2002galerkin}). Our choice of the Histogram method for density estimation will be explained in Sec.\ \ref{sec:spline_pdf}.

\section{Density-Estimation and Spline-based UQ}\label{sec:spline}

{\revdone Despite the prevalence of surrogate models in numerical methods and of density-estimation in UQ applications \cite{ablowitz2015interacting, chen2005uncertainty, colombo2018basins,  zabaras2007sparse, le2010asynchronous,  best2017paper, ullmann2014pod}, to the best of our knowledge, the adequacy of surrogate models for density estimation has not been addressed in the UQ literature. To study this problem, we first write an explicit relation between a function $f:\Omega \to \mathbb{R}$ and the PDF that it induces on $\mathbb{R}$:
\begin{lemma}\label{lem:pdf}
Let $f$ be a real, piecewise monotone, continuously differentiable function on $[a,b]$, where $-\infty\leq a <b \leq \infty$, and let $\mu$ be an absolutely continuous probability measure on $[a,b]$, i.e., there is $c\in L^1 \left( [a,b] \right)$ such that~$d\mu (\alpha) =~c(\alpha)d\alpha$. Then 
\begin{equation}\label{eq:pdf_inve}
p_f(y) = \sum\limits_{f(\alpha _j)=y} \frac{c(\alpha _j)}{|f'(\alpha _j)|} \, ,
\end{equation}
where $p(y)$ is the {\rm PDF} of $f$.
\end{lemma}
\begin{proof}
See Appendix \ref{app:pdf_pf}.
\end{proof}

Because polynomial approximations (e.g., gPC) tend to be oscillatory, they "add" many artificial extremal points. Hence, by Lemma \ref{lem:pdf}, the PDFs that they induce might deviate considerably from the exact one. To elucidate this point, in Lemma \ref{lem:counter_pdf} we consider a smooth function $f$ which is approximated by a highly oscillatory function~$g$. In this example, having an upper bound on $\|f-g\|_r$ for some $r\geq 1$ does not yield an upper bound on $\|p_f-p_g\|_q$, where $p_f$ and $p_g$ are the PDFs induced by $f(\alpha)$ and $g(\alpha)$, respectively, and $q\geq 1$, because of the numerous "artificial" extremal points of $g$. 
\begin{lemma}\label{lem:counter_pdf}
Let $\Omega = [0,1]$ equipped with the Lebesgue measure. Under the above notations, then for every $\epsilon >0$, there exists two functions $f$ and $g$ such that $\|f-g\|_{\infty} \leq \epsilon$, but $\|p_f-\tilde{p_g}\|_{\infty} \geq 1/2$.
\end{lemma} 
\begin{proof}
Let $f(\alpha) = \alpha$ and $g(\alpha)= \alpha ~+\delta\sin ((2\delta)^{-1}\alpha)$. By direct differentiation $g'(\alpha) = 1+~2^{-1} \cos((2\delta)^{-1}\alpha)$ and $f'(\alpha) \equiv 1$. Since $f$ is monotone, and since $g$ is monotone for sufficiently small $\delta$, then by Lemma \ref{lem:pdf} with $c(\alpha)\equiv 1$, and so $p_f(y)=1/f'(f^{-1}(y))\equiv 1$ and $p_g(y)=~1/g'(g^{-1}(y))$. Specifically, there exists $y\in \mathbb{R}$ such that $p_g(y)= 1/2$, and so $\|p_f-p_g\|_{\infty} \geq 1/2$, irrespective of $\|f-g\|_{\infty}=\delta$, which can be made arbitrarily small.
\end{proof} 
\remark A similar argument also shows that $\|f-g\|_r$ does not control $\|p_f-p_g\|_q$ for {\em any} $1\leq q,r \leq \infty$.

To propose a surrogate model for which accurate density-estimation is guaranteed, we first note that $f_N^{\rm gpc}(\alpha)$ is the interpolating polynomial of~$f$ of order $N-1$ at the Gauss quadrature points $\left\{ \alpha _j \right\} _{j=1}^{N}$ \cite{const2012sparse, hsu2017design}.}
This suggests that other interpolants of $f(\alpha)$ can be used in Algorithm \ref{alg:gpc_int}. In what follows, we argue that for our computational tasks, splines provide a better way to approximate~$f(\alpha)$ and its associated PDF.

We recall that splines are piecewise polynomials of degree $m$, with $k<m$ smooth derivatives. Given an interval $\Omega = [\alpha_{\min},\alpha_{\max}]$ and a grid $\alpha_{\min} = \alpha_{1} < \alpha _2 < \cdots <\alpha _N = \alpha_{\max}$, the interpolating cubic spline $s_N(\alpha)$ is a $C^2$, piecewise-cubic polynomial that interpolates~$f(\alpha)$~at~$\left\{ \alpha _j \right\} _{j=1}^N$, endowed with two additional boundary conditions. Three standard choices are {\revdone ({\em i}) The natural cubic spline, for which $\frac{d^2}{d\alpha^2}f_N^{\rm spline}(\alpha _1) = \frac{d^2}{d\alpha^2}f_N^{\rm spline}(\alpha _N) = 0$, ({\em ii}) The "not-a-knot" spline, for which $\frac{d^3}{d\alpha^3}f_N^{\rm spline}$ is continuous at $\alpha_2$ and $\alpha_{N-1}$, and ({\em iii}) The clamped spline, for which $\frac{d}{d\alpha}f_N^{\rm spline}(\alpha _j) = \frac{d}{d\alpha} f(\alpha_j)$ for $j=1,N$.}
Our decision to use splines is motivated by the following reasons:
\begin{enumerate}
\item  The error of spline interpolation is guaranteed to be "small" for {\em any} sample size, in the following sense:

\begin{theorem}[\cite{beatson1986spline, hall1976bounds}]\label{thm:hall_spl}
Let $f\in C^{m+1} \left( \left[ \alpha_{\min} , \alpha_{\max} \right] \right)$, and let $f_N ^{\rm spline}$ be its "not-a-knot", clamped or natural {\revdone $m$-th order} spline interpolant. Then
\begin{equation}
\big\|\big(f(\alpha) -f_N^{\rm spline}(\alpha) \big) ^{(j)} \big\|_{L^{\infty} [\alpha_{\min}, \alpha_{\max} ] } \leq C_{\rm spl} ^{(j,m)}\left\| f^{(m+1)} \right\| _{\infty} h_{\max} ^{m+1-j} \, , \qquad j=0,1,\ldots, m-1 \, ,
\end{equation}
where $C_{\rm spl} ^{(j,m)}>0$ is a universal constant that depends only on the {\em type} of boundary condition, $m$, and $j$, and $h_{\max} = \max\limits_{1<j\leq N} \lvert \alpha_j - \alpha _{j-1} \rvert$.
\end{theorem}

\item Spline interpolation is predominantly local. For further details, see Appendix \ref{ap:local}.

Thus, {\em although $f_N^{\rm spline}(\alpha)$ depends on $\left\{ f(\alpha _1), \ldots , f(\alpha_N) \right\}$, it predominantly depends on the few values $f(\alpha_j)$ for which $\alpha _j$ is adjacent to $\alpha$.} Therefore, large derivatives and discontinuities of $f(\alpha)$ may impair the accuracy of $f_N^{\rm spline}(\alpha)$ only locally.\footnote{For a review of cubic splines that are {\em strictly} local, see \cite{beatson1992cubic}.} This is in contrast to gPC (and polynomial interpolation in general), where discontinuities and large derivatives of $f$ decrease the approximation accuracy across the entire domain. In addition, splines can be constructed using any choice of sampling points.
\end{enumerate}



\noindent
In light of these considerations, we propose to replace the gPC interpolant with a spline:

\begin{algorithm}
\caption{Spline-based estimation}\label{alg:spline_u}
\begin{flushleft}
Let $\Lambda = \left\{\alpha _1 , \ldots, \alpha _N \right\}$ be a uniform grid on $[\alpha_{\min}, \alpha_{\max} ]$.
\end{flushleft}
\begin{algorithmic}[1]
\STATE{For each $\alpha _j \in \Lambda$, solve \eqref{eq:general_settings} with $\alpha = \alpha_j$ to obtain  $u \left(t, {\bf x}  ; \alpha _j  \right)$.}
\STATE{Approximate $u (t, x;\alpha ) \approx u_N^{\rm spline} (t, {\bf x};\alpha )$, where $u_N^{\rm spline}$ is a cubic spline interpolant on $\Lambda $. }

\STATE{Approximate $f(\tilde{\alpha} _m) \approx f(u _N ^{\rm spline}(\cdot, \tilde{\alpha } _m ) )$ on a sample of $M\gg N$ points $\left\{ \tilde{\alpha}_m \right\} _{m=1}^M$ which are i.i.d. according to $c(\alpha)$.}
\IF{{\bf goal is moment estimation:}}
\STATE{Use the trapezoidal integration rule with $\left\{ f(  \tilde{\alpha } _m ) \right\}_{m=1}^M$ to approximate $\mathbb{E}_{\alpha} [f]$.}
\ELSIF{{\bf goal is density estimation:}}
\STATE{Use the histogram method \eqref{eq:hist} with $\left\{ f(  \tilde{\alpha } _m ) \right\}_{m=1}^M$ to approximate the {\rm PDF} of $f$.}
\ENDIF
\end{algorithmic}
\end{algorithm}
\noindent \remark See Appendix \ref{app:code} for a MATLAB implementation of this algorithm.

Which cubic spline should be used in line 2? If $f'(\alpha_{\min})$~and $f'(\alpha_{\max})$ are known, then one should use the {\em clamped} cubic spline (or the natural cubic spline if these derivatives are zero). When the boundary derivatives are unknown, however, the "not-a-knot" interpolating cubic spline should be used (as indeed was done in this manuscript). See \cite{beatson1992cubic} for further discussion.

Algorithm \ref{alg:spline_u} is {\em identical} to Algorithm \ref{alg:gpc_int}, except for two substantial points:
\begin{enumerate}
\item The sampling grid is uniform, rather than the Gauss quadrature grid.\footnote{Algorithm \ref{alg:spline_u} can be performed with {\em any} choice of grid points. For clarity, we present it only with a uniform grid.}
\item The gPC interpolant $u_N^{\rm gpc}$ is replaced by a cubic spline interpolant $u_N^{\rm spline}$. 
\end{enumerate}

\remark This method is not to be confused with {\em spline-smoothing}, in which one approximates the PDF $p$ with splines \cite{ eubank1999nonparametric, wahba1990spline}. Thus, Algorithm \ref{alg:spline_u} approximates $u$ with a spline, but the resulting approximation of the PDF $p$ {\em is not a spline}. 

\subsection{Accuracy of Algorithm \ref{alg:spline_u} for density estimation}\label{sec:spline_pdf}

The density estimation error of Algorithm \ref{alg:spline_u} has two components - the error of the spline approximation (line 3) and that of the histogram method (line 7).\footnote{In terms of density estimators, this can be explained by the following argument. Denote by $p$, $p_N$, and $\hat{p}_{N,M}$ the density of $f$, $f_N$ and the density estimator of Algorithm \ref{alg:gpc_int} or \ref{alg:spline_u}, respectively. Then the approximation error (in any norm) satisfies $\|p-\hat{p}_{N,M}\| \leq \| p- p_N \| + \|p_N -\hat{p}_{N,M} \| $. The second term vanishes as $M\to \infty$ and $L$ is given by \eqref{eq:bins_Opt}, in which case the density estimation error is roughly the bias incurred from approximating $f$ by $f_N$.}

The accuracy of the histogram method in line 7 depends on the number of bins $L$ and on the number of samples $M$ at lines 3 and 7. If the number of bins is chosen to be
\begin{equation}\label{eq:bins_Opt}
L_{\rm opt} = K_f M^{-\frac{1}{3}} \, , \qquad K_f =\left(\frac{\|f'\|_2 ^2  [\max f - \min f]}{6} \right)^{\frac{1}{3}} \, , 
\end{equation} 
the mean squared $L^2$ error (MISE) of the histogram method decays as $M^{-\frac{2}{3}}$ \cite{wasserman2013all}.\footnote{In practice, $f$ and $f'$ are often unknown, and so $K_f$ needs to be estimated.} Because the computational cost of increasing $L$ and $M$ is negligible, they can be set sufficiently large so that the accuracy of Algorithm~\ref{alg:spline_u}~mainly depends on the difference between the PDFs of $f$ and $f_N ^{\rm spline}$, denoted by~$p_f$~and~$p_{f_N}$~respectively.
We motivate the choice of the histogram method to estimate the density by four factors:
\begin{enumerate}
\item Implementing the histogram method is straightforward, and can be done with a few lines of code (see Appendix \ref{app:code}).

\item The accuracy of the histogram method can be improved and controlled by varying the number of samples $M$, with a negligible computational cost.

\item The histogram method can be used even when the quantity of interest $f$ is not smooth.

\item The histogram method can be used for a multi-dimensional random parameter $\pmb{\alpha}$.
\end{enumerate}
In principle, we could have used the explicit relation \eqref{eq:pdf_inve} to compute the PDF. Because this approach does not have the above advantages, however, the histogram method was chosen.

\subsection{Accuracy of spline-based density estimation}

{\revdone In Section \ref{sec:spline_pdf} we showed that the accuracy of density estimation of Algorithms \ref{alg:gpc_int} and \ref{alg:spline_u} is determined by the error of approximating the density with that of the surrogate model, and not by the error of the histogram method. By Lemma \ref{lem:pdf}, if $f'(\alpha)$ is bounded away from zero, then $p$ is smooth.} As noted, however, the gPC polynomial interpolant $f_N^{\rm gpc}(\alpha)$ tends to be oscillatory, and so it might add artificial extermal points where $\frac{d}{d\alpha}f_N^{\rm gpc}(\alpha) =0$, see e.g., Fig. \ref{fig:tanh9_PDF}(c). At every such point where $\frac{d}{d\alpha}f_N^{\rm gpc}(\alpha) =0$, the PDF approximation becomes unbounded, and so a large error in the PDF estimation occurs. This is seldom the case with the spline interpolant, which due to its local nature (see Lemma \ref{lem:spline_local}) does not produce numerical oscillations throughout its domain $\Omega$. Indeed, the natural cubic spline $f_N^{\rm spline}(\alpha)$ has the "minimum curvature" property \cite{prenter2008splines}, which implies that it oscillates "very little" about the original function. This notion is made precise by the following result:
\begin{theorem}\label{thrm:cub_pdf}
Let $f\in C^{m+1} ([\alpha_{\min} , \alpha_{\max} ])$ with $\lvert f'(\alpha)\rvert \geq a >0$, let $\alpha$ be distributed by~$c(\alpha)d\alpha$, where $c \in C^1 \left( [\alpha_{\min}, \alpha_{\max}]\right)$, and let $p_f$ and $p_{f_N}$ be the {\rm PDF}s of $f(\alpha)$ and of $f_N = f_N ^{\rm spline}$, its natural, "not-a-knot", or clamped {\revdone $m$-th order} spline interpolant on a uniform grid of size $N$. {\revdone Then, for any $1\leq q < \infty$}
\begin{equation}\label{eq:thrm3}
\|p_f - p_{f_N}\|_q \leq KN^{-m} \, , \qquad  N> \sqrt[m]{\frac{2C_{\rm spl}^{(1,m)} \left\| f^{(m+1)} \right\|_{\infty}}{a}}\left(\alpha_{\max} - \alpha_{\min} \right) \, \, ,
\end{equation}
where $C_{\rm spl} ^{(1,m)}$ is given by Theorem {\rm \ref{thm:hall_spl}} and $K$ depends only on $f(\alpha)$, $c(\alpha)$, $q$, and $\lvert \alpha_{\max} - \alpha_{\min} \rvert$.\end{theorem}
\begin{proof}
See Appendix \ref{ap:cub_pdf_pf}.
\end{proof}
\noindent

{\revdone The proof of Theorem \ref{thrm:cub_pdf} only makes use of two properties of spline interpolation: the $L^{\infty}$ accurate approximationsional noise of the function and its derivative and the uniform bound of the second derivatives (Theorem \ref{thm:hall_spl}). Therefore, Theorem \ref{thrm:cub_pdf} immediately generalizes to a broad family of surrogate models, denoted by $\{g_N\}$:
\begin{corollary}\label{cor:1d_gen}
Let $f(\alpha)$ and $c(\alpha)$ be as in Theorem \ref{thrm:cub_pdf}, let $g_N\in C^1([\alpha_{\min},\alpha_{\max}])$ be a sequence of approximations of $f$ for which $$\|f-g_N\|_{\infty}\,  ,\,  \|f'-g_N'\|_{\infty} \leq KN^{-\tau} \, , \qquad \|g_N''\|_{\infty} < C_g < \infty \, ,$$ where $\tau > 0$, $C_g$, and $K$ are independent of $N$. Then $$\|p_f-p_{g_N}\|_q \leq \tilde{K}N^{-\tau} \, , $$ for any $1\leq q < \infty$, where $p_f$ and $p_{g_N}$ are the PDFs of $f(\alpha)$ and $g_N(\alpha)$, respectively, and $\tilde{K}$ is independent of $N$.
\end{corollary}}

\remark If $f$ is only piecewise $C^{m+1}$, then $N^{-m}$ convergence is guaranteed when the grid points include the discontinuity points of $f(\alpha)$, since the
proof can be repeated in each interval on which the function is $C^{m+1}$ in the same way.
\remark\label{rem:monotone} Although Theorem \ref{thrm:cub_pdf} applies only to functions whose derivatives are bounded away from $0$, in practice we observe cubic convergence for non-monotone functions as well~(see Sec.~\ref{sec:cnls}). Whether Theorem \ref{thrm:cub_pdf} generalizes to non-monotone cases is unclear.

In our numerical simulations, see Figs.\ \ref{fig:tanh9_PDF}, \ref{fig:tanh9_mod}, \ref{fig:polar_pdf}, and \ref{fig:Xs}, we observe that the cubic convergence is often reached well before $N$ satisfies \eqref{eq:thrm3}. We also observe that the density approximation error $\|p_f - p_{f_N}\|_1$ decays at a faster than cubic rate. A possible explanation for this observation is provided~by

\begin{lemma}\label{lem:quar}
Assume the conditions of Theorem {\rm\ref{thrm:cub_pdf}} for $m=3$, and let $J_N$ be the number of times that $\frac{d}{d\alpha}\big(f(\alpha)-f_N ^{\rm spline}(\alpha)\big)$ changes its sign on $[\alpha_{\min},\alpha_{\max}]$. If $J_N = O(N^{r})$ for $0\leq r \leq 1$, then $\|p_f - p_{f_N}\|_1 \leq KN^{-4+r}$. Specifically, if $J_N$ is uniformly bounded for all $N\in \mathbb{N}$, then $\|p_f-p_{f_N}\|_1 \leq KN^{-4}$.
\end{lemma}
\begin{proof}
See Appendix \ref{sec:quar}.
\end{proof}

\subsection{Accuracy of moment estimation}
{\revdone While the main focus of this paper is on density-estimation using surrogate model, we also point out two disadvantages of the gPC method for 
moment estimation:}
\begin{enumerate}
\item {\revdone The spectral convergence of the gPC method is attained only {\em asymptotically} as the number of sample points $N$ becomes sufficiently large. For small or moderate values of $N$,~however, its accuracy may be quite poor, due to insufficient resolution, and the global nature of spectral approximation.}

\item {\revdone The sample points $\left\{ \alpha _j \right\} _{j=1}^N$ of the gPC method are predetermined by the quadrature rule. Therefore, if one wants to {\em adaptively} improve the accuracy, one cannot use the samples from the "old" low-resolution grid in the "new" high-accuracy approximation.}

\end{enumerate}
\noindent
Similarly to density estimation, the error of the moment estimation of Algorithm \ref{alg:spline_u} comes from both the numerical integration (line 5) and interpolation (line 2). The trapezoidal rule integration error can be made sufficiently small by increasing the number of samples $M$ at line 3, at a negligible computational cost. Moreover, if $c(\alpha)\equiv 1$, the integration over $f_N ^{\rm spline}$ can be done exactly.\footnote{When $f$ is sufficiently smooth and $\alpha$ is uniformly distributed, one can approximate $\mathbb{E}_{\alpha} [f] \approx \mathbb{E}_{\alpha} [f_N ^{\rm spline}]$, and compute the right-hand side explicitly (in MATLAB, this can be done using the {\em fnint} command).} Hence, the moment estimation error of Algorithm \ref{alg:spline_u} is determined by the accuracy of the spline interpolation:

\begin{corollary}\label{corr:spl_mom}
Let $f\in C^{4} \left( [\alpha_{\min}, \alpha_{\max} ]\right)$, let $f_N ^{\rm spline}$ be the natural, "not-a-knot", or clamped cubic spline interpolant of $f$, and let $\alpha$ be distributed by $c(\alpha)d\alpha$, where $c(\alpha)\geq 0$, and $\int\limits_{\alpha _{\min}} ^{\alpha _{\max}} c(\alpha) \, d\alpha  = 1$. Then
$$\left| \mathbb{E}_{\alpha} [f] - \mathbb{E}_{\alpha} [f_N ^{\rm spline}] \right| \leq C_{\rm spl} ^{(0)} \|f\|_{\infty} h_{\max}^{4} \, ,$$
where $C_{\rm spl}^{(0)}$ and $h_{\max}$ are defined in Theorem {\rm \ref{thm:hall_spl}}.
\end{corollary}
\begin{proof}
By Theorem \ref{thm:hall_spl}, $\large\|f - f_N ^{\rm spline} \large\| _{\infty} \leq C_{\rm spl} ^{(0)} \large\|f^{(4)} \large\|_{\infty} h_{\max}^4$. Hence, 
\begin{multline*} 
\Big| \int\limits_{\alpha_{\min}} ^{\alpha _{\max}} \big( f(\alpha) - f_N ^{\rm spline}(\alpha) \big) c(\alpha)\, d \alpha \Big|\leq \\
  \big\|\big( f - f_N ^{\rm spline}\big) \big\|_{\infty} \int\limits_{\alpha_{\min}}^{\alpha_{\max}} c(\alpha) \, d\alpha \leq \left\|f - f_N ^{\rm spline} \right\| _{\infty} \cdot 1  \leq C_{\rm spl} ^{(0)} \|f^{(4)} \|_{\infty} h_{\max}^4 \, . 
\end{multline*}
\end{proof}
Typically, $C_{\rm spl} ^{(0)} <1$. For example, for the natural and "not-a-knot" cubic spline, $C_{\rm spl} ^{(0)}$ is equal to $\frac{5}{384}$ and $\frac{1}{25}$, respectively \cite{hall1976bounds, beatson1992cubic}. On a uniform grid, $h_j =\frac{\alpha_{\max} - \alpha_{\min}}{N-1}$ for $1<j\leq N$, and so  $ \mathbb{E}_{\alpha} \large[ f \large]- \mathbb{E}_{\alpha} \large[ f^{\rm spline} _N  \large]=~O(N^{-4})$.

As $N\to \infty$, the {\em polynomial} convergence rate of the spline approximation~(Corollary \ref{corr:spl_mom}) is outperformed by gPC's {\em spectral} convergence rate (Corollary \ref{corr:moment_gpc}). Quite often, however, the spline approximation is more accurate for {\em moderate} $N$ values. To see that, note that by \eqref{eq:orthogonal_polynomials_def}, \eqref{eq:coefs_with_quad}, and \eqref{eq:col_gpc}, $\mathbb{E}_{\alpha} \left[ f_N ^{\rm gpc} \right] =\sum\limits_{j=1}^{N} f(\alpha_{j} ) w_j $, which is the Gauss quadrature rule. Hence, if $f\in C^{2N}$, then 
$$  \mathbb{E}_{\alpha} \left[ f\right]- \mathbb{E}_{\alpha} \left[ f_N ^{\rm gpc} \right] = \frac{f^{(2N)} (\xi)}{k_N ^2 (2N)!} \, , \qquad \xi \in (\alpha_{\min} , \alpha_{\max}) \, ,$$
where $k_N$ is the leading coefficient of $p_N (\alpha)$~\cite{davis1967integration}. If for small $N$, $\|f^{(2N)}\|_{\infty}$ increases faster than $k_N ^2(2N)!$, the error initially {\em increases} with~$N$. In these cases, the exponential convergence is only achieved at large~$N$.\footnote{For example, if the numerator grows as $K^{2N}$, the error only decays for $N > K$.} Even when gPC does converge exponentially, i.e.,~$ \left| \mathbb{E}_{\alpha} \left[ f\right]- \mathbb{E}_{\alpha} \left[ f_N ^{\rm gpc} \right] \right| \leq Ke^{-\gamma N}$, if $\gamma$ is small, the error of the spline approximation may be smaller for moderate values of $N$, see e.g., Fig.\ \ref{fig:tanh9_func}(c). To conclude, the accuracy of spline-based moment approximation is {\em guaranteed} also with few samples, and not only asymptotically as $N\to \infty$.

\section{Multi-dimensional noises}\label{sec:multid_theory}

To generalize the spline-based density-estimation approach (Algorithm \ref{alg:spline_u}) to the case where $\pmb{\alpha} \in \Omega =[0,1]^d$, we use tensor-product splines, which are defined in the following way. Let $m\geq 1$, let $f(\pmb{\alpha})\in C^{m+1}(\Omega)$, let $\Lambda$ be the one-dimensional grid $0= \alpha _1 < \cdots < \alpha _n =~1$, and let $\Lambda ^d$ be the respective $d$-dimensional tensor-product grid. An $m$-th degree {\em tensor-product spline interpolant} of $f$ is a function $s(\pmb{\alpha})\in C^{m-1} (\Omega)$ that interpolates $f$ on $\Lambda ^d$ and reduces to a {\em one-dimensional} $m$-th degree spline on every line on $\Lambda ^d$, \footnote{i.e., when $d-1$ coordinates of~$\pmb{\alpha}$~are fixed in $\Lambda$.}\textsuperscript{,}\footnote{$s(\pmb{\alpha})$ is unique when endowed with sufficiently many boundary conditions, see the discussion on the one-dimensional case in Sec.\ \ref{sec:spline}. Theorem \ref{thm:tensor_spline} holds for many possible choices of boundary conditions, including the not-a-knot conditions which we have also used in our simulations.} see \cite{schultz1969spline} for a more precise definition. The multidimensional extension of Algorithm \ref{alg:spline_u} for density estimation is

\begin{algorithm}
\caption{Multidimensional spline-based density estimation}\label{alg:spline_multid}
\begin{flushleft}
Let $\Lambda ^d = \left\{ \alpha _1 , \ldots ,\alpha _N\right\}^D$ be a tensor-product uniform grid on $[0,1 ]^d$.
\end{flushleft}
\begin{algorithmic}[1]
\STATE{For each $\pmb{\alpha}_{\bf j} \in \Lambda ^d$, solve \eqref{eq:general_settings} with $\alpha = \pmb{\alpha}_{\bf j}$ to obtain  $u \left(t, {\bf x}  ; \pmb{\alpha} _{\bf j } \right)$.}
\STATE{Approximate $u (t, x;\pmb{\alpha} ) \approx u_N^{\rm spline} (t, {\bf x};\pmb{\alpha} )$, where $u_N^{\rm spline}$ is a tensor-product spline interpolant of order $m$ on $\Lambda ^d$. }

\STATE{Approximate $f(\tilde{\pmb{\alpha}} _m) \approx f(u _N ^{\rm spline}(\cdot, \tilde{\pmb{\alpha} } _m ) )$ on a sample of $M\gg N$ points $\left\{ \tilde{\pmb{\alpha}}_m \right\} _{m=1}^M$ which are i.i.d. according to $c(\pmb{\alpha})$.}
\STATE{Use the histogram method \eqref{eq:hist} with $\left\{ f(  \tilde{\pmb{\alpha} } _m ) \right\}_{m=1}^M$ to approximate the {\rm PDF} of $f$.}

\end{algorithmic}
\end{algorithm}

As in the one-dimensional Algorithm \ref{alg:spline_u}, the analysis of the density-estimation error in Algorithm \ref{alg:spline_multid} is based on two components:
\begin{enumerate}
\item A pointwise error bound for tensor-product spline interpolants, due to Schultz:
\begin{theorem}[\cite{rice1978spline, schultz1969spline}]\label{thm:tensor_spline}
Let $\Omega = [0,1]^d$, $f\in C^{m+1}(\Omega)$, and let $s(\pmb{\alpha})$ be its $m$-th degree tensor-product spline interpolant. Then for any $\pmb{\alpha} \in \Omega$,
\begin{equation}\label{eq:multid_spl_bd}
|D^{j} (f-s)\big|<C_m h^{m+1-j} \, , \qquad j=0,1, \ldots m-1 \, ,
\end{equation}
where $D^j$ is any $j$-th order derivative,\footnote{More explicitly, $D^j=\prod_{k=1}^{d}(\partial _{\alpha _{k}})^{\ell _k}$ where $\ell _1 + \cdots \ell _d = j$, and each $\ell _k$ is a non-negative integer.} $C_m = C_m (\| D^{m+1} f \| _{\infty})$ depends only on the $L^{\infty}$ norms of the $m+1$ order derivatives of $f$, and $h = \max _{1\leq j <n} |\alpha_{j+1} - \alpha_j |$.
\end{theorem}

\item A multi-dimensional generalization of Lemma~\ref{lem:pdf}.\footnote{When $\Omega \subset \mathbb{R}$ is a one-dimensional interval, Lemma \ref{lem:ddim_pdf} reduces to Lemma \ref{lem:pdf}. Indeed, since $|f'| \neq 0$ on $\bar{\Omega}$ then $f$ is piece-wise monotonic, and so $f^{-1}(y)$ consists of a finite number of points. In addition, the surface element $d\sigma$ is a point-mass distribution. Hence, \eqref{eq:multid_pdf}~reduces to \eqref{eq:pdf_inve}.}

\begin{lemma}
\label{lem:ddim_pdf}
Let $\Omega \subset \mathbb{R}^d$ be a Jordan set, denote by $|\cdot |$ the Euclidean norm in $\mathbb{R}^d$, let~$f$~be piecewise-differentiable with $|\nabla f| \neq 0$ on $\bar{\Omega}$, let $\pmb{\alpha}$ be an absolutely-continuous random variable in $\Omega$, i.e., $d\mu (\pmb{\alpha}) = c(\pmb{\alpha})d\pmb{\alpha}$ for some non-negative $c\in  L^1 (\Omega)$, and denote the PDF associated with $f(\pmb{ \alpha})$ by $p_f$. Then 
\begin{equation}\label{eq:multid_pdf}
p_f(y) = \frac{1}{\mu(\Omega)}\int\limits_{f^{-1} (y)} \frac{c(\pmb{\alpha})}{| \nabla f (\pmb{\alpha})|} \, d\sigma \, ,
\end{equation}
where $d\sigma$ is a $(d-1)$ dimensional surface element on~$f^{-1}(y)$.
\end{lemma}
\begin{proof}
See Appendix \ref{ap:ddim_pdf_pf}.
\end{proof}

\end{enumerate}

The generalization of Theorem \ref{thrm:cub_pdf} to the case of multidimensional random parameter is as follows:
\begin{theorem}\label{thm:multid}
Let $\Omega = [0,1]^d$, let $m\geq 1$, let $f\in C^{m+1}(\Omega)$, let $s$ be the $m$-degree tensor-product spline interpolant of $f$, let $\pmb{\alpha}$ be uniformly distributed in $\Omega$, and let $p_f$ and $p_s$ be the PDFs of $f$ and~$s$,~respectively. If~$\kappa _{\rm f} :=~\min _{\Omega} |\nabla f|  >0$, then for sufficiently small $h$ {\revdone and for any $1\leq q<\infty$}, 
\begin{equation}\label{eq:bd_pdf_multid}
\|p_f - p_s\|_q \leq Kh^m  \, ,
\end{equation}
for some constant $K>0$, where $h$ is defined in Theorem \ref{thm:tensor_spline}.
\end{theorem}
\begin{proof}
See Appendix \ref{app:multid_pdf_pf}.
\end{proof}

Theorem \ref{thm:multid} can be extended to any approximation $\tilde{f}$ of $f$ and to any bounded domain $\Omega \subseteq \mathbb{R}^d$, provided that the bound \eqref{eq:multid_spl_bd} holds for $j=0$ and $j=1$.

The total number of sample points in the special case where $\Lambda$ is the uniform one-dimensional grid on $[0,1]$ is $N=n^d \sim  h^{-d}$. Therefore,
\begin{corollary}\label{corr:pdf_multid_N}
Let $\Lambda$ be the uniform grid on $[0,1]$. Then under the conditions of Theorem \ref{thm:multid}, then for sufficiently large $N$, $$\|p_f - p_s\|_1 \leq  KN^{-\frac{m}{d}}  \, ,$$ for some constant $K>0$.
\end{corollary}

As noted in Sec.\ \ref{sec:kde}, the $L^1$ error of the KDE method scales as $N^{-\frac{2}{5}}$ \cite{devroye1985l1}. Therefore, by Corollary \ref{corr:pdf_multid_N}, Algorithm \ref{alg:spline_multid} outperforms KDEs for dimensions $d\leq \frac{5}{2}m$. {\revdone Finally, as in the one-dimensional case (Corollary \ref{cor:1d_gen}), the proof of Theorem \ref{thm:multid} only makes use of two properties of spline interpolation: the $L^{\infty}$ approximation of the function and of its gradient, and the uniform bound on the second derivatives (Theorem~\ref{thm:tensor_spline}). Theorem~\ref{thm:multid} therefore generalizes immediately to density estimation using non-spline surrogate models:
\begin{corollary}\label{cor:gen_multid}
Under the conditions and notations of Theorem \ref{thm:multid}, consider $g_h\in C^1{[0,1]^d}$ with uniformly bounded second derivatives such that $$\|f-g_h\|_{\infty} \, , \, \|\nabla f - \nabla g_h\|_{\infty} \leq Kh^{-\tau} \, ,$$
for some $\tau>0$ independent of $f$ and $K=K(f)>0$. Then $\|p_f-p_{g_h}\|_{q} \leq \tilde{K}h^{-\tau}$ for any $1\leq q < \infty$.
\end{corollary}
}

\section{Simulations}\label{sec:example}\label{sec:tanh}

In this section, we compute the density and the moments of the function
\begin{equation}\label{eq:tanh}
f(\alpha) = \tanh (9\alpha) + \frac{\alpha}{2} \, ,\qquad \alpha \in [-1,1] \, \, ,
\end{equation}
which is smooth but has a narrow high-derivative region.\footnote{The $\frac{\alpha}{2}$ term was added so that $\frac{df}{d\alpha}$ is bounded away from zero, in order to prevent singularities in the PDF, see Sec.\ \ref{sec:pdf_tanh}}

\subsection{Interpolation}

\begin{figure}[h!]
\centering
{\includegraphics[scale=0.57]{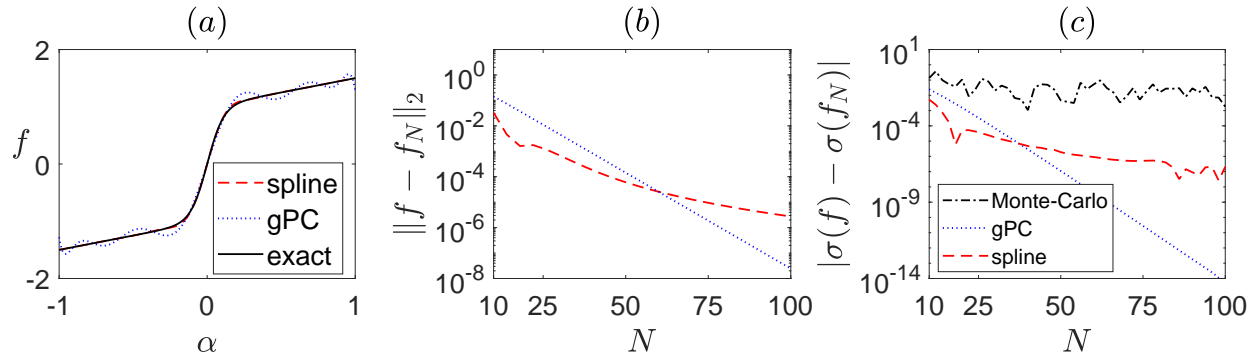}}
\caption{(a) $f(\alpha)$ (solid), see \eqref{eq:tanh}, and its spline interpolant (dashes) are nearly indistinguishable, whereas the gPC interpolant (dots) oscillates "around" $f$. Both interpolants use $N=12$ grid points. (b) $L^2$ error of both interpolants as a function of the number of samples. (c)~Error of the standard deviation when it is approximated using Monte-Carlo method (dash-dot), the gPC-based method (dots) and the spline-based method (dashes).}
\label{fig:tanh9_func}
\end{figure}
With $N=12$ samples, the spline interpolant $f_N ^{\rm spline}$ of \eqref{eq:tanh} is nearly indistinguishable from~$f$,~whereas the gPC interpolant $f_n^{\rm gpc}$ slightly oscilates "around" $f$, see Fig.\ \ref{fig:tanh9_func}(a). Although $f_N^{\rm gpc}$ converges exponentially to $f$ in $L^2$, see Fig.\ \ref{fig:tanh9_func}(b), its $L^2$ approximation error~$\left\|f-f_N\right\|_2 =\big(\int_{-1}^{1} \left| f(\alpha)- f_N (\alpha) \right|^2 \, d\alpha \big)^{\frac{1}{2}} $
with few samples ($10\leq N\leq 40$) is larger than that of the spline interpolant by more than an order of magnitude. With sufficiently many samples~($N>70$), however, the gPC approximation exponential convergence outperforms the spline's polynomial convergence rate. This example shows that with few samples, the occurrence of a "jump" in $f$ hurts the accuracy of the gPC interpolant. Spline interpolation, on the other hand, is less sensitive to the "jump", because it "confines" the approximation error induced by the jump to the jump interval (roughly $\alpha \in (-0.1,0.1)$), see Lemma \ref{lem:spline_local}.

\subsection{Moment approximation} 
The interpolation accuracy is relevant to moment approximation, because a small $L^2$ error implies a small moment-approximation error (Lemma \ref{lem:l2bounds}). For example, Fig.~\ref{fig:tanh9_func}(c)~shows the standard deviation error $\left|\sigma (f) -\sigma (f_N)\right|$, see \eqref{eq:tanh}, when $\alpha$ is uniformly distributed in~$[-1,1]$. As expected, the spline-based method (Algorithm \ref{alg:spline_u}) is more accurate than the gPC-based method (Algorithm~\ref{alg:gpc_int}) with few samples, but the gPC is more accurate with sufficiently many samples. A purely statistical approach such as Monte-Carlo converges poorly compared to both the spline and gPC approach, with about $ 10\% $ error with $N\leq 100$ sample points. 


\subsection{Density estimation}\label{sec:pdf_tanh}

\begin{figure}[h!]
\centering
{\hspace{-1cm} 	\includegraphics[scale=0.55]{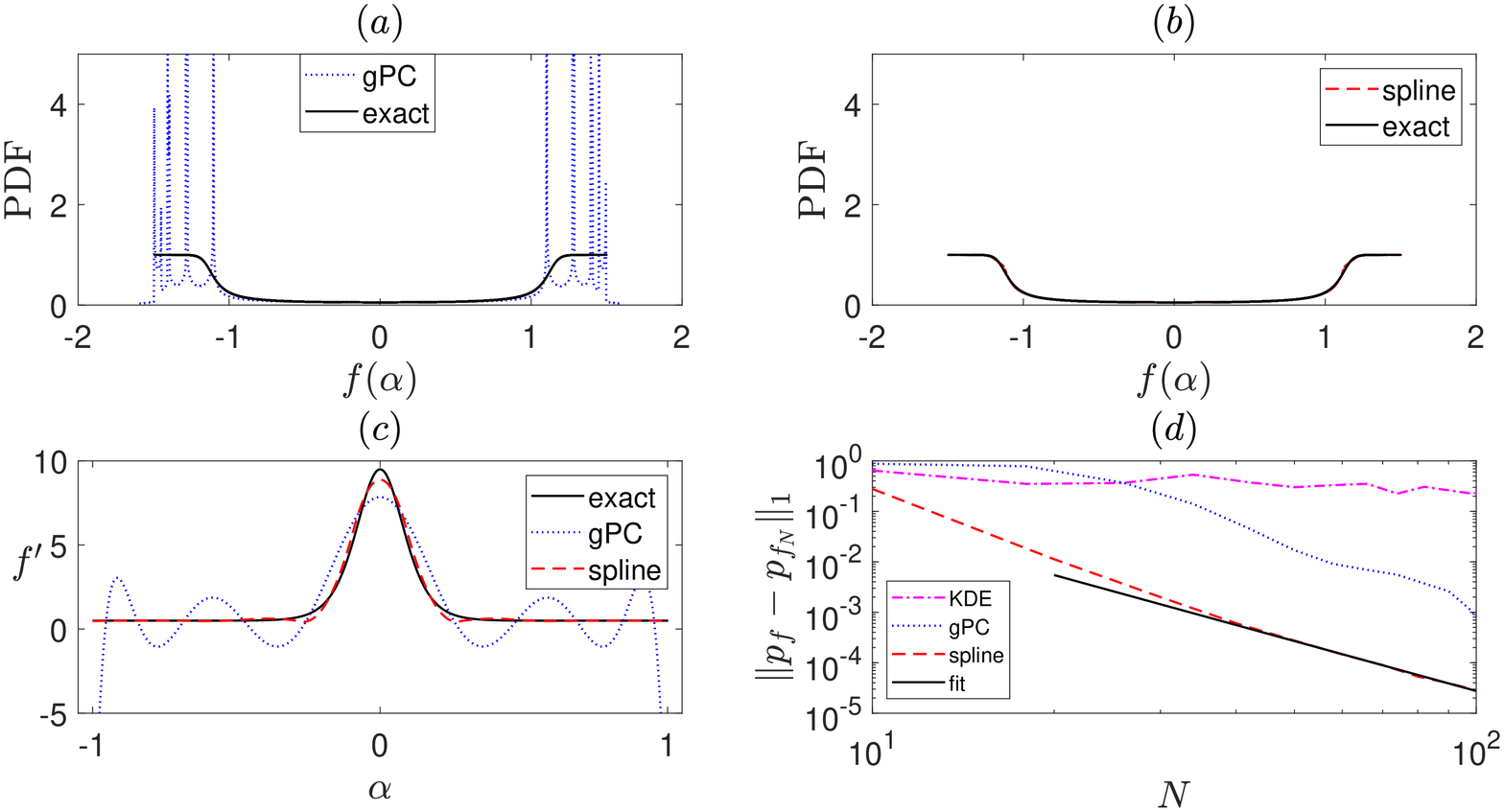}}
\caption{The PDF of $f(\alpha)$, see \eqref{eq:tanh}, where $\alpha$ is uniformly distributed in $[-1,1]$. (a)~exact PDF (solid) and its approximation by the gPC-based Algorithm \ref{alg:gpc_int} (dots) with $N=18$ sample points. (b) Same, with the spline-based Algorithm \ref{alg:spline_u} (dashes). The two lines are nearly indistinguishable. (c) Derivatives of $f$ (solid), $f_N ^{\rm spline}$ (dashes) and $f_N ^{\rm gpc}$ (dots). (d) $L^1$ error of the PDF approximations as a function of the number of sample points, for the KDE~(dash-dot), gPC-based approximation (dots), the spline-based approximation~(dashes), and its power-law fit $103.2 N^{-3.29}$ (solid).}
\label{fig:tanh9_PDF}
\end{figure}

Consider the PDF induced by $f(\alpha)$, see \eqref{eq:tanh}, when $\alpha$ is uniformly distributed in $[-1,1]$. The PDF computed by the gPC-based Algorithm \ref{alg:gpc_int} with $N=18$ sample points deviates considerably from the exact PDF, see Fig.\ \ref{fig:tanh9_PDF}(a), whereas the PDF computed by the spline-based Algorithm \ref{alg:spline_u} with $N=18$ sample points is nearly indistinguishable from the exact PDF, see~Fig.~\ref{fig:tanh9_PDF}(b).\footnote{The MATLAB code that generates this PDF approximation is given in Appendix \ref{app:code}.} This is consistent with our discussion in Sec.\ \ref{sec:spline}. Indeed, the derivative of the spline interpolant $\frac{d}{d\alpha}f_N ^{\rm spline}$ approximates $f'(\alpha)$ with cubic accuracy, whereas the derivative of the gPC interpolant $\frac{d}{d\alpha}f_N ^{\rm gpc}$ has many artificial extremal points where $\frac{d}{d\alpha}f_N ^{\rm gpc} (\alpha) = 0$, but $\frac{d}{d\alpha}f(\alpha) \neq 0$, see Fig.\ \ref{fig:tanh9_PDF}(c). 

The $L^1$ distance $\|p_f - p_{f_N}\|_1$ between the exact PDF~$p_f$~and its approximation~$p_{f_N}$ is presented in Fig.\ \ref{fig:tanh9_PDF}(d). For $10 \leq N \leq 100$ the spline-based approximation is more accurate than the gPC-based one by nearly two orders of magnitude. This is in contrast to moment estimation, see Fig.\ \ref{fig:tanh9_func}(c), in which the gPC approximation becomes more accurate for $N \geq 40$. Furthermore, we observe numerically that the spline-based method converges even faster than the $N^{-3}$ rate predicted by Theorem~\ref{thrm:cub_pdf}. The KDE approximation has roughly $10\%$ error for $N\leq 100$.\footnote{The poor accuracy of the KDE method is due to the fact that the KDE does not use the "functional information" $\left\{f_j =f(\alpha _j )\right\}_{j=1}^N$, but only the set $\left\{ f_j\right\}_{i=1}^N$.} Other frequently-used distances between distributions, such as the Hellinger distance $\frac{1}{\sqrt{2}} \left\| \sqrt{p_f}- \sqrt{p_{f_N}} \right\|_2 $ \cite{le2012asymptotics} and the Kullback-Leibler~(KL) Divergence\footnote{Intuitively, the $d_{\rm KL}$ measures the entropy added, or conversely, the information lost, in approximating~$p$~by~$p_{f_N}$.} 
 \cite{kl1951original}
\begin{equation}\label{eq:kl}
\int\limits_{-\infty}^{\infty} p(y)\log \left(\frac{p_f(y)}{p_{f_N}(y)} \right) \, dy \, \, ,
\end{equation}
produce similar results (data not shown).    
\subsection{Density estimation of non-smooth functions}\label{sec:nonsmooth}
Let 
\begin{equation}\label{eq:tanh_mod}
g(\alpha) =  f(\alpha)  \, {\rm mod} \, (0.7)  \, ,
\end{equation}
where $f$ is given by \eqref{eq:tanh}.\footnote{This example is motivated by our study of the NLS \cite{best2017paper}, where the cumulative phase~$\varphi(t;\alpha ) =~{\rm arg} \,\left[ \psi(t,0;\alpha) \right]$ is smooth, but the quantity of interest, the angle $\varphi \, {\rm mod} \, (2\pi )$, is discontinuous. See Sec.\ \ref{sec:cnls}. for another optics application which motivates this example.} Because \eqref{eq:tanh_mod} is non-smooth, with few samples neither the spline, nor the gPC interpolant are even remotely close to $g(\alpha)$, see Fig.\ \ref{fig:tanh9_mod_direct}. Therefore, to approximate the PDF associated with $g(\alpha)$, we first use Algorithms \ref{alg:gpc_int} and \ref{alg:spline_u} to approximate~$f(\alpha)~\approx~f_N(\alpha)$. Since $f$ is smooth, both approximations are reasonable with few samples, see Fig.~\ref{fig:tanh9_func}. Next, we approximate  $g(\alpha_m) \approx  f_N(\alpha_m) \, {\rm mod} \, (0.7)$, and compute the PDF of~$g$~using the histogram method on a high-resolution sampling grid ($M=2\cdot 10^6$). We again stress that evaluating~$f_N$~is computationally cheap, and therefore can be easily done with such a large sample. As in the smooth case, see Fig.~\ref{fig:tanh9_PDF}, the PDF approximated by the gPC-based Algorithm \ref{alg:gpc_int} with $N=18$ sample points has large deviations and converges poorly, see Fig.\ \ref{fig:tanh9_mod}(a), whereas the PDF approximated by the spline-based Algorithm~\ref{alg:spline_u} with~$N=18$ sample points is nearly identical to the exact PDF, see Fig.\ \ref{fig:tanh9_mod}(b). Indeed the $L^1$ error of spline-based PDF is smaller than that of the gPC-based PDF by at least an order of magnitude, for $20<N<50$, see Fig.~\ref{fig:tanh9_mod}(c). Although Theorem \ref{thrm:cub_pdf} applies only to~$C^4$~functions, we observe numerically that the convergence rate of the spline-based PDF is faster than $N^{-3}$. The KDE approximation for the PDF of~$g(\alpha)$~is less accurate than that of the spline-based and gPC-based approximations.

\begin{figure}[h!]
\centering
{ \includegraphics[scale=0.55]{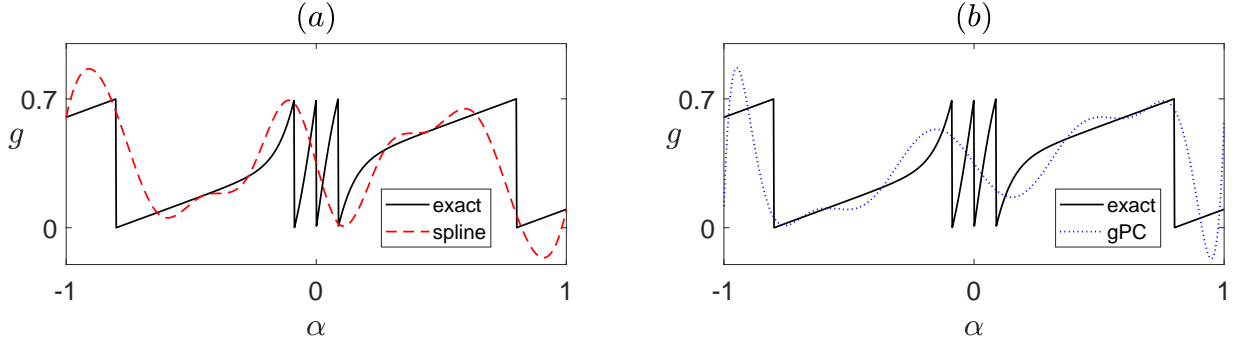}}
\caption{The discontinuous function $g(\alpha)$, see \eqref{eq:tanh_mod} (solid), and its spline interpolation with $N=12$ sample points (dashes). (b) Same with the gPC interpolant (dots).}
\label{fig:tanh9_mod_direct}
\end{figure}

\begin{figure}[h!]
\centering
{\hspace{0cm} \includegraphics[scale=0.55]{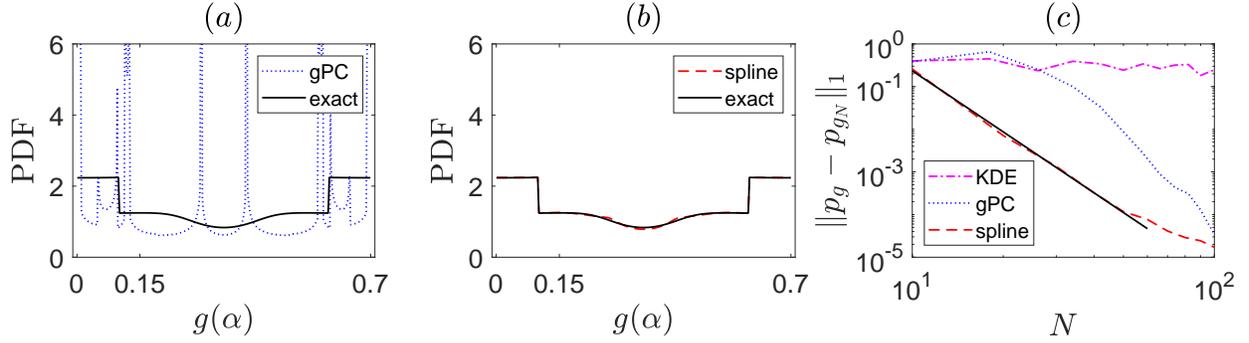}}
\caption{Same as Fig.\ \ref{fig:tanh9_PDF} for the discontinuous function $g(\alpha)$, see \eqref{eq:tanh_mod}. The solid line in subplot (c) is the power-law fit $1.33 \cdot 10^4 N^{-4.75}$ of the spline-based approximate PDF.}
\label{fig:tanh9_mod}
\end{figure}

\subsection{Multidimensional noise}\label{sec:multid_toy}

To numerically confirm the error bound of the density estimation (Algorithm \ref{alg:spline_multid}) for $d>1$, we first consider the two-dimensional function
\begin{equation}\label{eq:2d_tanh}
f_{\rm 2d}(\alpha_1, \alpha_2) = \tanh (6\alpha _1 \alpha _2 + \alpha _1 /2) + (\alpha _1 + \alpha _2 )/3 \, .
\end{equation}
where $\alpha _1$ and $\alpha_2$ are independent and uniformly distributed in $[-1,1]$. As in the one-dimensional example, see \eqref{eq:tanh}, $f_{\rm 2d}$ is analytic with high-gradients regions, see Fig.~\ref{fig:2d_tanh}(a). The spline-based PDF approximation with $N=8^2$ sample points is very close to the exact PDF of $f(\alpha _1 , \alpha _2)$, whereas the gPC-based PDF deviates from it substantially (Fig.~\ref{fig:2d_tanh}(b)). The convergence rate of Algorithn \ref{alg:spline_multid} with cubic splines is~$N^{-2.15}$ (Fig.\ \ref{fig:2d_tanh}(c)), which is consistent with the theoretical $N^{-\frac{3}{2}}$ error bound (Corollary \ref{corr:pdf_multid_N}). The convergence rates of both the KDE and the gPC methods are considerably slower for "small" sample sizes ($N\leq 200$).
\begin{figure}[h!]
\centering
{ \includegraphics[scale=0.4]{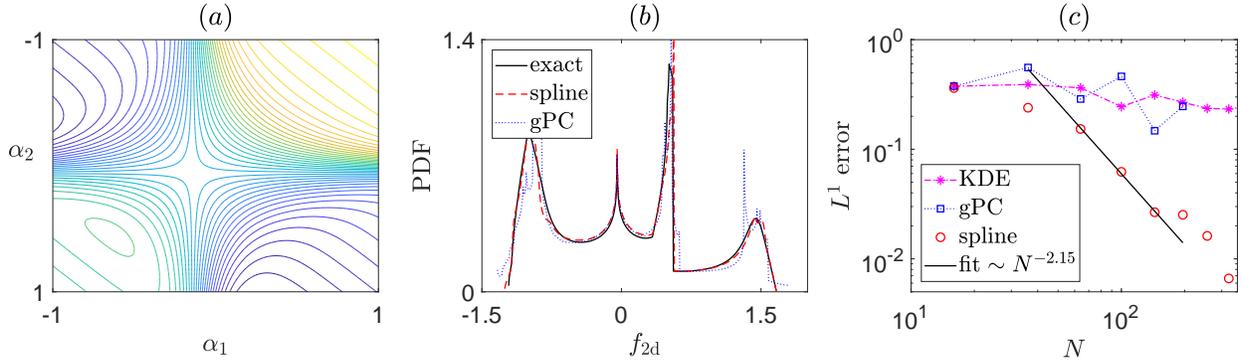}}
\caption{(a) Contours of the function $f_{\rm 2d}(\pmb{\alpha})$, see \eqref{eq:2d_tanh}. (b) The PDF of $f_{\rm 2d}(\pmb{\alpha})$ (solid), its approximation by the spline-based Algorithm \ref{alg:spline_u} (dashes), and by the gPC-based Algorithm \ref{alg:gpc_int} (dots). Here $\pmb{\alpha}$ is uniformly distributed in~$[-1,1]^2$, and both approximations use $N=64$ sample points. (c) $L^1$ error of the PDF approximations as a function of the number of sample points, for the KDE~(dash-dots), gPC-based approximation (dots-squares), the spline-based approximation~(circles). The solid line is the power-law fit $1208 N^{-2.15}$ (solid).}
\label{fig:2d_tanh}
\end{figure}

Next, consider the three-dimensional function
\begin{equation}\label{eq:3d_tanh}
f_{\rm 3d}(\alpha_1 ,\alpha_2 ,\alpha _3) = \tanh(8\alpha _1 +5\alpha _2 + 10\alpha _3) + (\alpha _1 +\alpha_2 +\alpha _3)/3 \, ,
\end{equation}
where $\alpha _1$, $\alpha _2$, and $\alpha _3$ are independent and uniformly distributed in $[-1,1]$. The spline-based PDF with $N=10^3$ sample points approximates the exact PDF well, see Fig.\ \ref{fig:3d_tanh}(a), and its convergence rate is $N^{-1.1}$ (see Fig.\ \ref{fig:3d_tanh}(b)), which is consistent with the theoretical $N^{-1}$ convergence rate (Corollary \ref{corr:pdf_multid_N}). For comparison, the fitted convergence rate of the KDE is $N^{-0.39}$, which is consistent with the theoretical $N^{-\frac{2}{5}}$ rate \cite{devroye1985l1}. Therefore, the spline-based method is more accurate than the KDE for sufficiently many samples ($N>10^3$). For smaller values of $N$ (e.g., $N=216$), however, the KDE achieves a slightly better accuracy than the spline-based method. This can be explained by what is known as the "curse of dimensionality". Thus, in the three-dimensional tensor-grid spline, $N=216$ sample points correspond to a mere {\em six} sample points in each dimension, which leads to insufficient resolution. The KDE method, on the other hand, does not approximate the underlying function $f_{\rm 3d}$, and is therefore "indifferent" to the noise dimension. See Sec.\ \ref{sec:final} for further discussion.
\begin{figure}[h!]
\centering
{ \includegraphics[scale=0.4]{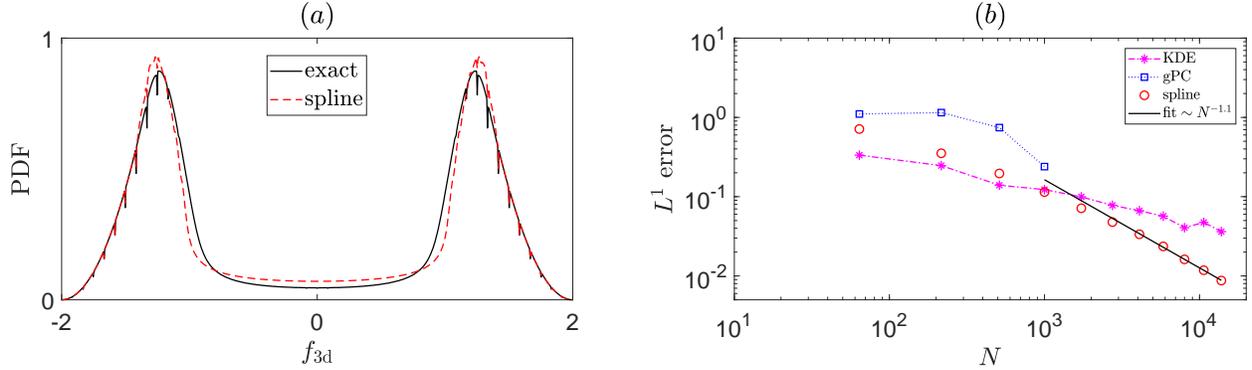}}
\caption{(a) The PDF of $f_{\rm 3d}(\pmb{\alpha})$, see \eqref{eq:3d_tanh}, where $\pmb{\alpha}$ is uniformly distributed in~$[-1,1]^3$ (solid) and its approximation by the spline-based Algorithm \ref{alg:spline_u} (dashes) with $N=8^3$ sample points. (b) $L^1$ error of the PDF approximations as a function of the number of sample points, for the KDE~(dash-dots), the gPC-based PDF(rectangles), the spline-based PDF~(circles), and its power-law fit $354N^{-1.11}$ (solid).}
\label{fig:3d_tanh}
\end{figure}

\section{Application 1 - Nonlinear Schr{\"o}dinger equation}\label{sec:cnls}

The one-dimensional coupled nonlinear Schr{\"o}dinger equation~(CNLS)
\begin{equation}\label{eq:cnls}
i\frac{\partial A_{\pm}(t,x)}{\partial t} + \frac{\partial ^2 A_{\pm}}{\partial x^2} + \frac{2}{3}\frac{\left|A_{\pm}\right|^2 + 2\left|A_{\mp} \right|^2}{1+\epsilon \left( \left|A_{\pm}\right|^2 + \left|A_{\mp} \right|^2 \right)}A_{\pm} = 0 \, ,
\end{equation}
where $0<\epsilon \ll 1$, $t\geq 0$, and $x\in \mathbb{R}$, describes the propagation of elliptically polarized, ultra-short pulses in optical fibers \cite{agrawal2007book}, of elliptically polarized continuous-wave (CW) beams in a bulk medium \cite{patwardhan2017loss, sheinfux2012measuring}, Stokes and anti-Stokes radiation in Raman amplifiers \cite{randoux2011intracavity}, and rogue water-waves formation at the interaction of crossing seas \cite{ablowitz2015interacting}. 
We consider \eqref{eq:cnls} with an elliptically-polarized Gaussian input pulse with a random amplitude  \cite{patwardhan2017loss, sheinfux2012measuring}
\begin{equation}\label{eq:rnd_elliptic_ic}
\left(\begin{array}{l}
A_+ \\ A_-
\end{array} \right) = \left(1+0.1\alpha\right)\left(\begin{array}{l}
8 \\ 4
\end{array}\right) e^{-x^2} \, ,
\end{equation}
where $A_+$ and $A_-$ are the clockwise and counter-clockwise circularly-polarized components, respectively. The {\em on-axis ellipse rotation angle} is defined as
\begin{equation}\label{eq:theta}
\theta (t;\alpha) :\,= \left(\varphi _+ (t;\alpha) - \varphi _- (t;\alpha)\right) ~ {\rm mod} \, (2\pi) \, , \end{equation}
where $\varphi _{\pm}(t;\alpha)  :\,= {\rm arg} \left[A_{\pm} (t,0;\alpha)\right]$ are the on-axis phases of the components. The distribution of $\theta(t;\alpha)$ indicates to what extent the ellipse rotation angle is "deterministic".\footnote{We solve the CNLS using a fourth-order, compact finite-difference scheme for the spatial discretization, and a predictor-corrector Crank-Nicolson scheme for the temporal integration of the semi-discrete problem~\cite{fibich2015nonlinear}.}

\begin{figure}[h!]
\centering
{\hspace{-1.2cm}\includegraphics[scale=0.53]{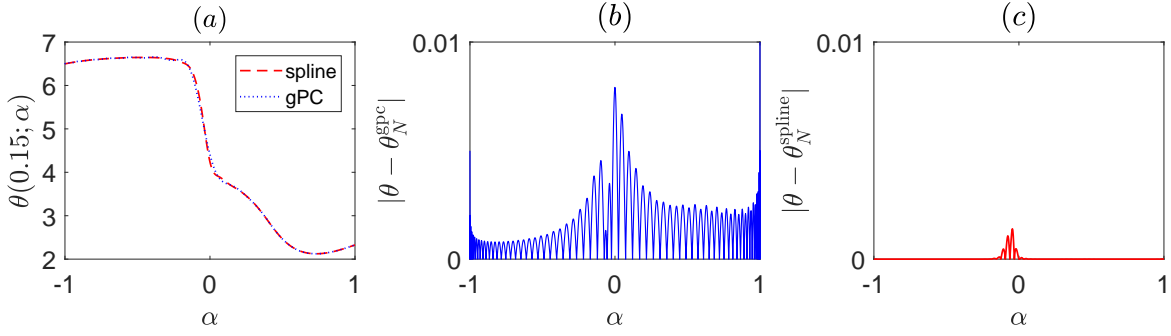}}
\caption{The polarization angle $\theta(t= 0.15;\alpha)$ for solutions of the CNLS \eqref{eq:rnd_elliptic_ic} with $\epsilon=10^{-5}$, and an elliptically polarized Gaussian initial condition \eqref{eq:rnd_elliptic_ic}. (a) Spline interpolation (dashes) and gPC interpolation (dots), with $N=64$ sample points. The two lines are nearly indistinguishable. (b) Pointwise error of the gPC interpolant. (c) Same for the spline interpolant. }
\label{fig:polar}
\end{figure}

\paragraph{Interpolation} For a given sample grid $\left\{\alpha_j\right\}_{j=1}^N$, we compute $\theta (t;\alpha_j)$ for each $1\leq j \leq N$ by solving \eqref{eq:cnls}--\eqref{eq:rnd_elliptic_ic} and using~\eqref{eq:theta}. Fig.\ \ref{fig:polar}(a) shows the spline and gPC interpolants of $\theta(t= 0.15;\alpha)$ with~$N=~64$ points.\footnote{Because we have no explicit solution for $\theta(t;\alpha)$, the errors in this section are measured by comparison with $\theta_{513} ^{\rm spline}(0.15,\alpha)$ with $N=513$ sample points. We verified that $\big\| \theta_{513} ^{\rm spline}(0.15,\alpha) -~\theta_{513} ^{\rm gpc}(0.15,\alpha)\big\| _2 \approx~5\cdot~10^{-5}$, which is an order of magnitude smaller than the approximation errors noted in the text.} While these interpolants seem nearly identical, the spline interpolant is more accurate than the gPC interpolant by more then an order of magnitude (cf.\ Figs.\ \ref{fig:polar}(b) and~\ref{fig:polar}(c)). Indeed, the $L^2$ error of the gPC interpolant ($0.17\%$) is an order of magnitude larger than that of the spline interpolant ($0.017\%$).


\paragraph{Density estimation}  The gPC-based approximation with $N=64$ differs substantially from the exact PDF, see Fig.\ \ref{fig:polar_pdf}(a). In contrast, the spline-based approximated PDF with $N=64$ sample points is indistinguishable from the exact PDF, see Fig.\ \ref{fig:polar_pdf}(b). Indeed, the KL divergence of the gPC-based approximation, see \eqref{eq:kl}, is {\em about $16,000$ times larger} than that of the spline-based approximation, and the $L^1$ error is $200$ times larger (${ 46\%}$ vs.\ ${ 0.2\%}$). With $N=32$, the spline-based {\em is $32$ times more accurate} than the gpc-approximated PDF, in term of KL divergence, and $11$ time more accurate in terms of the $L^1$ error (${ 41\%}$ vs.\ ${ 4.5\%}$). The $L^1$ error of the spline-based PDF decays as $N^{-3.76}$, see Fig.\ \ref{fig:polar_pdf}(c). {\revdone This results "exceeds expectations" with respect to Theorem~\ref{thrm:cub_pdf}, since $\theta'(0.15;\alpha)$ is not bounded away from $0$ (see Fig. \ref{fig:polar}(a)), and so Theorem \ref{thrm:cub_pdf} should not, in principle, apply to this case.} Since the PDF of $\theta(0.15;\alpha)$ has discontinuities and high derivatives, spline smoothing techniques and KDE methods with smooth kernels were not considered in this case.

\begin{figure}[h!]
\centering
{\includegraphics[scale=0.5]{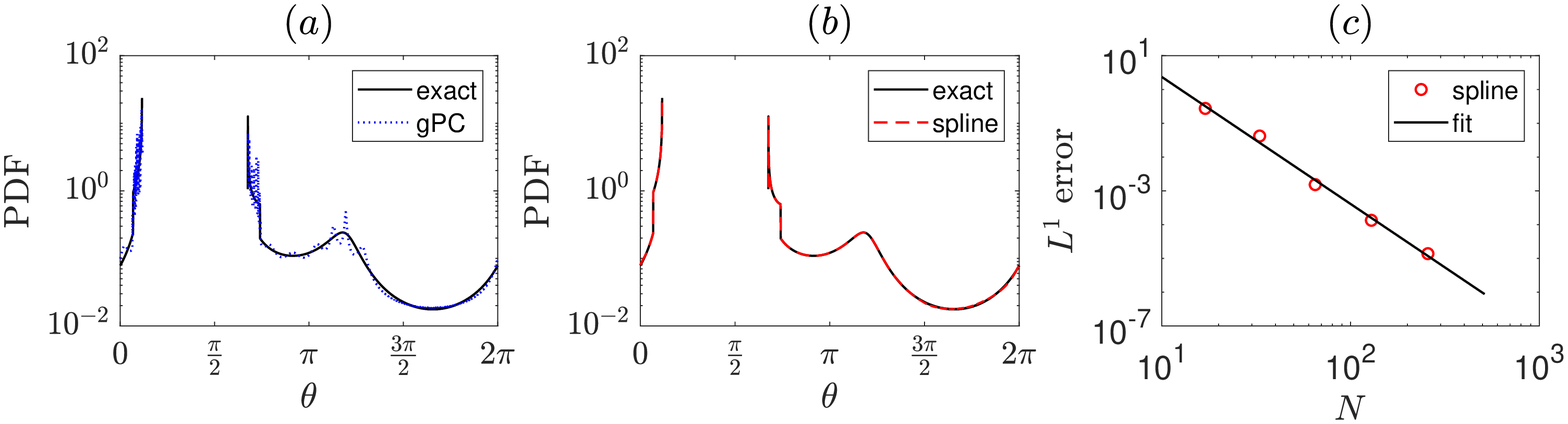}}
\caption{Same settings as in Fig.\ \ref{fig:polar}. The PDF of $\theta(0.15,\alpha)$, where~$\alpha \sim~U(-1,1)$. (a)~Exact PDF (solid), and gpc-based approximation using $N=64$ sample points (dots). (b) Same with the spline-based approximation~(dashes). The two lines are indistinguishable. (c) $L^1$ error of the spline-based PDF as a function of $N$ (circles) and the power-law fit $1.35\cdot 10^4 N^{-3.76}$ (solid).}
\label{fig:polar_pdf}
\end{figure}

\paragraph{Moment approximation} The mean and standard deviation of circular quantities can be defined as \cite{mardia2009directional}\footnote{To motivate why a different definition for circular moments is needed, consider $y\sim U(-\pi, \pi)$ and $z\sim U(0,2\pi)$. If we consider $y$ and $z$ as {\em angles}, or points on the circle, they are identical. Using the conventional mean definition, however, yields $\mathbb{E}[y] = 0$, but $\mathbb{E}[z]=\pi$.}
\begin{equation}\label{eq:circ}
\mathbb{E}_{\alpha}^{\rm circ} [\theta(t;\alpha)] = \int\limits_{-1}^{1} e^{i\theta(t;\alpha)} \, d\alpha \, , \qquad \sigma ^{\rm circ} (\theta) = \sqrt{-2{\rm ln} \, \left|\mathbb{E}_{\alpha}^{\rm circ}[\theta(t;\alpha)]\right|} \, .
\end{equation}
The advantage of splines over gPC with few samples for moments approximation can be seen in Table \ref{tab:circ_stats}. The approximation of $\mathbb{E}_{\alpha}^{\rm circ}[\theta(0.15;\alpha)]$ using the spline approximation with~$N=32$ is~$4$~times more accurate than that of the gPC; with~$N=64$~it is~$14$~times more accurate. The approximation of the standard deviation using the spline-based method with $N=32$ is $12$ times more accurate than the gPC; with $N=64$ it is $33$ times more accurate than the gPC-based approximation.

\begin{table}[h]
\center
\begin{tabular}{|c | c||c|c|c|}
\hline
\quad              & $N$ & gPC error & spline error & \footnotesize $ \frac{\text{gPC error}}{\text{spline error}}$ \normalsize \\
\hline
$\mathbb{E}_{\alpha} ^{\rm circ} \left[ \theta(0.15; \alpha) \right] $ & $ 32$ &  $2.2\%$ & $0.54\%$ & $4$   \\
$\mathbb{E}_{\alpha} ^{\rm circ} \left[ \theta(0.15; \alpha) \right] $ & $ 64$             & $0.089\%$ & $0.006\%$ & $14$   \\ \hline
$\sigma ^{\rm circ} \left( \theta(0.15; \alpha) \right) $ &$ 32$            & $0.64\%$ & $0.054\%$ & $12$  \\
$\sigma ^{\rm circ} \left( \theta(0.15; \alpha) \right) $ & $64$               & $0.031\%$ & $0.0009\%$ & $33$  \\
\hline
\end{tabular}
\caption{Approximation error of the circular mean and standard deviation, see \eqref{eq:circ}, of~$\theta(0.15,\alpha)$, see \eqref{eq:theta}, with gPC- and spline-based approximations, using $N$ sample points.}
\label{tab:circ_stats}
\end{table}


\section{Application 2 - inviscid Burgers equation}\label{sec:burgers}
The inviscid Burgers equation

\begin{equation}\label{eq:burgers}
u_t (t,x) + \frac{1}{2}(u^2)_{x} = \frac{1}{2}(\sin ^2 (x) )_x \, , \qquad x\in [0,\pi] \, ,\quad t\geq 0 \, ,
\end{equation}
with the initial and boundary conditions $ u(0,x) = u_0 (x)$ and $u(t,0)=u(t,\pi) =0$ models isentropic gas flow in a dual-throat nozzle. Solutions of this equation can develop a static shock wave at a lateral location~$x =X_{\rm s}$ \cite{salas1986multiple}. Following \cite{chen2005uncertainty}, we consider the case in which $\alpha$ is a random variable with a known distribution, $u_0(x)=u_0(x;\alpha)$ is random, and we wish to compute the PDF of $X_{\rm s} $ using Algorithms \ref{alg:gpc_int} and \ref{alg:spline_u}. In general, to do that requires, for each~$1\leq j \leq N$, to compute $X_{\rm s} (\alpha_j)$ by solving~\eqref{eq:burgers}~with $\alpha_j$. For the special initial condition
\begin{subequations}
\begin{equation}
u_0 (x) = \alpha \sin (x) \, ,
\end{equation}
however, the shock location is explicitly given by \cite{chen2005uncertainty}
\begin{equation}\label{eq:Xs_exp}
\alpha= -\cos (X_{\rm s}) \, \, .
\end{equation}  
\end{subequations} 
This explicit expression allows us to sample $X_{\rm s} (\alpha)$ without solving \eqref{eq:burgers}.

Consider the case where
\begin{equation}\label{eq:beta_n}
\alpha = \left\{ \begin{array}{ll}
\frac{-1+\sqrt{1+4\nu ^2}}{2\nu} \,  & {\rm if}~\nu\neq 0 \, , \\ 0 \, & {\rm if}~ \nu =0 \, ,
\end{array} \right. \,
\end{equation} and $\nu\sim \mathcal{N}(0,\sigma)$, i.e., it is normally distributed with a zero mean. Because $\alpha$ is not distributed by a classical, standard measure, there is no obvious choice of quadrature points to sample by, nor is there a "natural" orthogonal polynomials basis to expand the solution by. Therefore, the gPC approach cannot be straightforwardly applied.\footnote{Nevertheless, even for non-standard distributions, the expansion of~$\alpha$ by a classical orthogonal-polynomials basis can still converge spectrally, under certain conditions \cite{ditkowski2016gpc}.} We can, however, apply the gPC approach to this problem by denoting $X_{\rm s} (\nu) =~X_{\rm s} (\alpha (\nu))$, and approximating $X_{\rm s} (\nu)$ using the Hermite polynomials (which are orthogonal with respect to the normal distribution).\footnote{Indeed, in \cite{chen2005uncertainty} the authors use the gPC-Galerkin method with the Hermite polynomials \cite{ghanem2003stochastic, xiu2002galerkin}.} The gPC-based approximated PDF with $N=7$ sample points differs considerably from the exact PDF, see Fig.\ \ref{fig:Xs}(a). In contrast, the spline-based approximated PDF can be directly applied to $X_{\rm s} (\alpha)$, and it is nearly indistinguishable from the exact PDF already with $N=7$ sample points, see Fig.\ \ref{fig:Xs}(b). In general, the spline-based PDF approximation is more accurate than the gPC-based approximation by more than one order of magnitude for~$5<N<50$, see Fig.\ref{fig:Xs}(c). The $L^1$ error of the spline-based PDF is observed numerically to decay as $N^{-3.11}$, in accordance with Theorem \ref{thrm:cub_pdf}.


We repeated these simulations for the case with $\alpha \sim B(r,s)$, where $B(r,s)$ is the Beta distribution on $[-1,1]$.\footnote{The PDF of the Beta distribution on $[0,1]$ is $p(\alpha) = \frac{(\alpha ^{r-1} \left( 1-\alpha \right) ^{s-1} \Gamma(r+s)}{\Gamma (r) \Gamma (s)}$.} The spline based approximations are nearly identical to the exact PDF, whereas the gPC method were less accurate by an order of magnitude with few samples (results not shown).

\begin{figure}[h!]
\centering
{\includegraphics[scale=0.55]{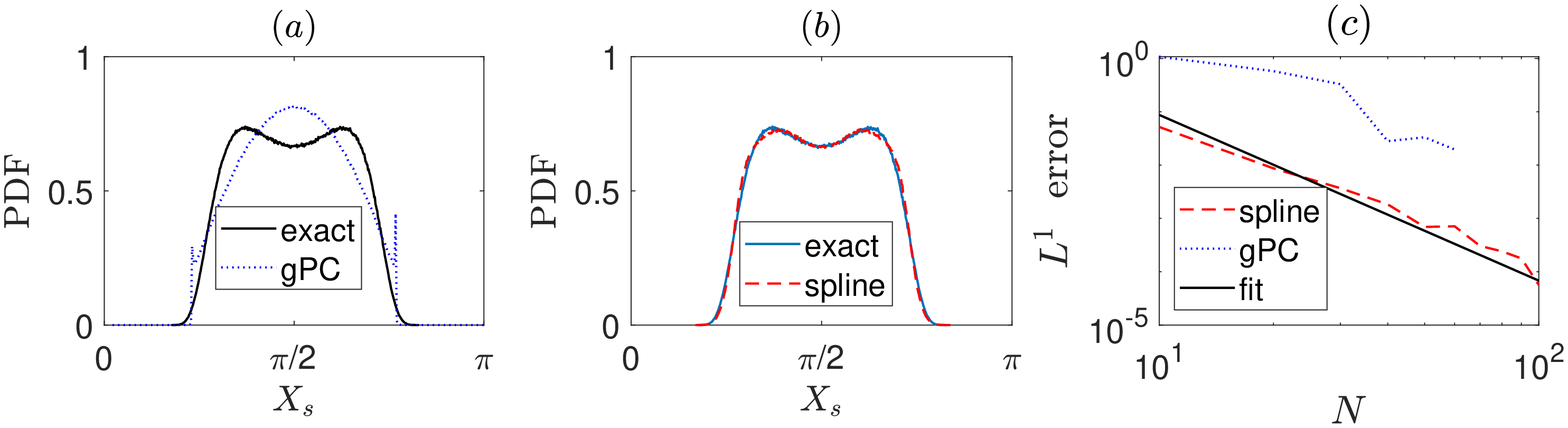}}
\caption{PDF of $X_s(\alpha)$, where $\alpha(\nu)$ is given by \eqref{eq:beta_n}, and $\nu \sim \mathcal{N}(0,0.6)$. (a) Exact PDF (solid) and gPC-based approximation (dots) with $N=7$ sample points. (b) Same with the spline-based approximation (dashes). (c) $L^1$ error of the PDF approximations as a function of the number of sample points, and the power-law fit $112N^{-3.11}$ (solid).}
\label{fig:Xs}
\end{figure}

\section{Discussion}\label{sec:final}

In this paper, we introduced a spline-based method for density and moment estimation. The advantages of this method are:
\begin{enumerate}
\item Our $m$-th order spline-based method approximates the density at a guaranteed convergence rate of $N^{-\frac{m}{d}}$, where $N$ is the sample size and $d$ is the noise dimension. Thus, our method outperforms KDEs for noise dimensions $1\leq d\leq \frac{5}{2}m$.
\item It provides reasonable approximations for the density and moments using small sample sizes. 

\item Its accuracy is relatively unimpaired by the presence of large derivatives.
\item It is non-intrusive, i.e., it is based solely on solving the underlying deterministic model.
\item It is easy to implement.
\item It is applicable with many choices of sample points.
\item It can be applied to non-smooth quantities of interest.
\end{enumerate}
\noindent

When $f\in C^{m+1}$, it is tempting to use splines of order $m>3$ for density estimation, in order to attain faster than cubic convergence rate. If one generalizes Algorithm \ref{alg:spline_u} to splines of order $m$ then, similarly to Theorem \ref{thrm:cub_pdf}, a convergence of order $N^{-m}$ is guaranteed. Even if $f$ is analytic, however, it is not advisable to take a large $m$, for two reasons. First, for $s(\alpha)$ to be monotone (and so, by Lemma \ref{lem:pdf} for the PDF to be continuous), $N$ should scale as $\sqrt[m]{\|f ^{(m+1)} \|_{\infty} }$, see~\eqref{eq:s_mono}. Therefore, for a large $m$, high-order convergence might only be attained for very large sample sizes. Second, the density approximation error depends linearly on $\left\| f^{(m+1)} \right\| _{\infty}$, see \ref{ap:cub_pdf_pf}, and so it might "blow-up" exponentially with~$m$. To conclude, although we do not know whether the optimal spline order is $m=3$, an arbitrarily high-order spline should not be used.

When approximating a $d$-dimensional function with a resolution $h$ at each dimension, the total number of samples $N$ scales as $h^{-d}$. As a result, for a prescribed accuracy, the computational cost grows exponentially with the dimension (the "curse of dimensionality"). In other words, for a given $N$, the accuracy decays exponentially with the dimesnion. Indeed, this is consistent with the $N^{-\frac{m}{d}}$ error estimate of the spline-based Algorithm \ref{alg:spline_multid} (Corollary \ref{corr:pdf_multid_N}). In contrast, the KDE method, which is a standard nonparametric statistical density estimator, converges at a rate of $N^{-\frac{2}{5}}$, regardless of $d$. Hence, our method will outperform KDE for "low" dimensions ($d <\frac{5}{2}m$), but may become inferior to KDE at higher dimensions. 

A popular approach for moment estimation of high-dimensional noise is the use of sparse sampling grids \cite{ghanem2017handbook, xiu2010numerical}. Recently, a spline approximation based on sparse grids was used in the context of forward uncertainty propagation \cite{halder2018adaptive}. Most sparse-grid methods, however, are designed with moment estimation in mind. As we have seen, even in the one-dimensional case (see Sec.\ \ref{sec:spline_pdf}), an accurate moment approximation does not necessarily imply an accurate density estimation. Whether sparse-grids methods can be adapted to density estimation remains an open question. the proof of Theorem \ref{thm:multid} in Appendix \ref{app:multid_pdf_pf}, however, suggests sufficient conditions by which new approximation methods can be tested for efficient density estimation: (1) The settings should be such that Lemma \ref{lem:ddim_pdf} applies, and (2) the approximation method should have a pointwise error bounds similar to Theorem \ref{thm:tensor_spline}.

In this paper we showed that spline-based density estimation is better than gPC-based density estimation, because it does not produce numerous artificial extremal points (see Lemma \ref{lem:pdf}). An interpolating cubic spline, however, might still produce artificial extremal points, though not as much as the gPC polynomial. To absolutely prevent artificial extremal points from being produced, it may be better to use spline interpolants \cite{fritsch1980monotone} and quasi-interpolants~\cite{deboor1978splines}~which are {\em monotonicity-preserving} (i.e., splines which are monotone wherever the sampled data is monotone). Hence, although these methods have the same {\em order} of error (with respect to $h$) as spline interpolation, they may provide better approximations for small samples, as they are {\em guaranteed} not to produce artificial extremal points. We leave it to future research to check whether monotonicity-preserving interpolants provide more accurate PDF approximations than a standard interpolating cubic spline.

{\revdone As noted throughout the paper, the $L^{\infty}$ error bounds on the quantity of interest and its gradient are key for the success of our algorithm, see Corollaries \ref{cor:1d_gen} and \ref{cor:gen_multid}. Since locality plays an important role in the existence of such error bounds for splines, it is natural to explore the use of other local approximations such as NURBS \cite{piegl2012nurbs, turner2009nurbs} and Radial Basis functions (RBF) \cite{fornberg2015rbf, schaback1995rbf}. An additional improvement may be achieved by designing surrogate models that are on one hand local, but on the other hand supported on an unbounded domain, e.g., Gaussian Mixtures \cite{terejanu2008gauss}. While moment-approximation in the case of unbounded input random parameters (e.g., normally or exponentially distributed $\pmb{\alpha}$) are theoretically well understood, the rigorous study of density-estimation in these setting is left for future research.}

\section{Acknowledgments}
The authors thank Y.\ Harness, B.\ Brill, R.\ Kats, F.\ Abramovich, and D.\ Levin for useful comments and conversations.

\appendix
\section{Proof of Corollary \ref{corr:moment_gpc}}\label{app:l2bounds_pf}
We begin with the following Lemma:
\begin{lemma}\label{lem:l2bounds}
Let $(\Omega , \mu )$ be a probability space, denote $\| \cdot \|_p :\,=\|\cdot \|_{L^p(\Omega)}$, and let $f,g \in L^2  \cap L^1 $. Then
\begin{subequations}
\begin{align}
&\left| \mathbb{E}_{\alpha} [f] - \mathbb{E}_{\alpha} [g] \right| \leq \|f - g\|_2 \, \label{eq:mean_err} , \\
&\left| {\rm Var}(f) - {\rm  Var}(g) \right| \leq  (\sigma(f) +\sigma(g) ) \cdot \|f-g\|_2 \, \label{eq:var_err}  ,\\
&\left|\sigma(f)-\sigma(g) \right| \leq    \|f-g\|_2 \, . \label{eq:std_err}
\end{align}
\end{subequations}
\end{lemma}
\begin{proof}

For all $f,g \in L^2$,
$$\Bigl\lvert \mathbb{E}_{\alpha} [f] - \mathbb{E}_{\alpha}[g] \Bigr\rvert \leq \int_{\Omega} \left| f(\alpha)-g(\alpha) \right| \, d\mu (\alpha) =\int_{\Omega} 1\cdot \left| f(\alpha)-g(\alpha) \right| \, d\mu (\alpha)  \leq \|1\|_2 \cdot \|f-g\|_2 = \| f - g\|_2 \, ,$$
where in the second inequality we used the Cauchy-Schwarz inequality. Thus, we proved~\eqref{eq:mean_err}.

For $h\in L^2 \cap L^1$, let $\tilde{h} :\,= h-\mathbb{E}_{\alpha} [h]$. By definition, ${\rm Var} (h) = \|\tilde{h}\|_2 ^2$ and $\sigma (h) = \|\tilde{h}\|_2$.~Hence,
\begin{equation}\label{eq:almost_var}
\begin{split}
\left| {\rm Var}(f) - {\rm Var}(g) \right| &= \left| \mathbb{E}_{\alpha} [\tilde{f} ^2 - \tilde{g} ^2] \right| = \left|\int\limits_{\Omega} (\tilde{f} - \tilde{g}) (\tilde{f}+\tilde{g} )\, d\mu (\alpha) \right| \leq  \|\tilde{f}  + \tilde{g}\|_2 \cdot \|\tilde{f} - \tilde{g} \| _2 \\
&\leq   \left( \|\tilde{f}\|_2  + \|\tilde{g}\|_2 \right)\cdot \|\tilde{f} - \tilde{g} \| _2 = \left(\sigma(f) + \sigma(g) \right) \cdot \|\tilde{f} -\tilde{g}\|_2 \, .
\end{split}
\end{equation}
In addition, $\|\tilde{h} \|_2 ^2  ={\rm Var} (h) = \mathbb{E}_{\alpha} [h^2] - \mathbb{E}_{\alpha} ^2 [h] \leq \mathbb{E}_{\alpha} [h^2] = \|h\|_2 ^2 \, , $ and so $\|\tilde{h}\|_2 \leq \|h\|_2$. Applying this inequality with $h=f-g$ to \eqref{eq:almost_var} yields \eqref{eq:var_err}. Finally, by \eqref{eq:var_err},

$$\bigl\lvert \sigma(f)-\sigma(g) \bigr\rvert = \left|\frac{\sigma ^2(f) - \sigma ^2 (g)}{\sigma (f) +\sigma (g)} \right|= \frac{\left|{\rm Var}(f) - {\rm  Var}(g)\right|}{\left| \sigma (f) +\sigma (g)\right|}\leq \frac{ \sigma (f) +\sigma (g)}{\sigma(f) +\sigma(g)} \|f-g\|_2 = \|f-g\|_2 \, .$$
which proves \eqref{eq:std_err}.
\end{proof}

In the case of gPC, let $g=f_N^{\rm gpc}$, the colocation gPC approximation of $f$, see \eqref{eq:col_gpc}. Since $f_N^{\rm gpc}$ converges exponentially to $f$ in the $L^2$ norm \cite{xiu2010numerical, gotlieb2007spectral}, Lemma \ref{lem:l2bounds} implies that the moments of $f_N ^{\rm gpc}$ converge exponentially to the moments of $f$. 

\section{Local properties of spline interpolation}\label{ap:local}
Let us first recall a classical result of Birkhoff and de Boor:
\begin{theorem}[\cite{deboor1964decay, deboor1976vanish}]\label{thm:deboor}
Let $s_i(\alpha)$ be the natural cubic spline that satisfies $s_i (\alpha _k) = \delta_{i,k}$, where $1\leq i,k \leq N$ and $\alpha_{\min} = \alpha_1
<\alpha_2 < \cdots <\alpha _N =\alpha_{\max}$ is given. Then $$\max\limits_{\alpha \not\in (\alpha_{i-k}, \alpha_{i+k})} \lvert s_i (\alpha) \rvert \leq A2^{-k} \, ,\qquad 1<i<N \, ,$$
where $A>0$ is a constant that depends on the global mesh ratio $\frac{\max\limits_{1<j\leq N} \alpha_j - \alpha _{j-1}}{\min\limits_{1<k\leq N} \alpha_k - \alpha _{k-1}}$.
\end{theorem}
Therefore, the natural cubic spline  $f_N^{\rm spline}(\alpha)$ is essentially a local approximation:
\begin{corollary}\label{lem:spline_local}
Denote the natural cubic spline $f_N^{\rm spline}=f_N^{\rm spline}(\alpha;f_1,\ldots , f_N)$ to emphasize the dependence of the spline interpolation on the sampled values. Then  $$\max\limits_{\alpha \not\in (\alpha _{i-k}, \alpha _{i+k} )} \Big\lvert \frac{\partial f_N^{\rm spline}(\alpha;f_1,\ldots , f_N)}{\partial f_i} \Big\rvert \leq~A2 ^{-k} \, , \qquad 1<i<N  \, ,\quad 1\leq k\leq N \, ,$$
where $A>0$ is given by Theorem \ref{thm:deboor}.
\end{corollary}
\begin{proof}
The function $S(\alpha) = \sum\limits_{i=1}^N f_i s_i(\alpha)$, where $s_i (\alpha)$ are defined in Theorem \ref{thm:deboor}, is a $C^2$ cubic spline, which by definition satisfies $S(\alpha_i) = f_i$, and $\frac{d}{d\alpha}S(\alpha_1)=~\frac{d}{d\alpha}S(\alpha_N)=~0$. By the uniqueness of the natural cubic spline, $S(\alpha) = f_N ^{\rm gpc} (\alpha)$, so, $\frac{\partial f_N^{\rm spline}(\alpha;f_1,\ldots , f_N)}{\partial f_i} =~s_i(\alpha)$. Hence, by Theorem \ref{thm:deboor}, the corollary is proven.
\end{proof}

\section{Proof of Lemma \ref{lem:pdf}}\label{app:pdf_pf}

When $f$ is strictly increasing, its CDF is given by
$$
P_f(y) :\,= \int\limits_{\alpha _{\min} } ^{f^{-1}(y)} c(\alpha) \, d\alpha \, .$$
By the Leibniz rule and the inverse function theorem,
$$
p_f (y) = \frac{dP_f(y)}{dy} = c \left( f^{-1}(y)\right) \left( f^{-1}\right)' = c\left( f^{-1}(y)\right) \frac{1}{f' \left( f^{-1} (y)\right)} \, .
$$
Similarly, if $f$ is monotonically decreasing, then $P_f(y) = \int_{f^{-1} (y)}^{\alpha_{\max}} c(\alpha) \, d\alpha $, and so $$p_f (y)=~-~\frac{c(f^{-1}(y))}{f'(f^{-1}(y))} \, .$$ Note that since $f' <0$, then $p_f(y)\geq 0$. Finally, if $f$ is piecewise monotonic, we apply this method separately on each sub-interval on which it is monotonic, and sum up the contributions.

\section{Sample MATLAB code for Algorithm \ref{alg:spline_u}}\label{app:code}
The following MATLAB code generates the dashed curve in Fig.\ \ref{fig:tanh9_PDF}(b).
\lstset{language=Matlab,%
    breaklines=true,%
    morekeywords={matlab2tikz},
    keywordstyle=\color{blue},%
    morekeywords=[2]{1}, keywordstyle=[2]{\color{black}},
    identifierstyle=\color{black},%
    stringstyle=\color{mylilas},
    commentstyle=\color{mygreen},%
    showstringspaces=false,
    numbers=left,%
    numberstyle={\tiny \color{black}},
    numbersep=9pt, 
    emph=[1]{for,end,break},emphstyle=[1]\color{red}, 
}
\begin{lstlisting}
alpha_min = -1; alpha_max =  1; N = 18;    %sample size
f = @(x) tanh(9*x) + x/2; 
    %define the initial sample on the grid [alpha_1, ... ,alpha_N]
samplingGrid = linspace(alpha_min,alpha_max, N);
samples = f(samplingGrid); % step 1
    %define the refined sample grid [tilde_alpha_1, ... tilde_alpha_M]
M = 2e6; 
denseGrid = linspace(alpha_min,alpha_max,M); 
fN_spline =  spline(samplingGrid,samples,denseGrid);       % steps 2+3
    %When f is given explicitly, the optimal number of bins (L)
    %is given by (14)
Cf = 1.69; L =Cf*M^(1/3);
    %step 4 - histogram of fN on denseGrid, not normalized
[histogram,binsEdges] = hist(fN_spline,L);
binWidth = (max(binsEdges)-min(binsEdges))/L; 
    %normalize the histogram so that it would be a PDF
pdf = histogram/(sum(histogram)*binWidth); 
plot(binsEdges,pdf)
\end{lstlisting}

\section{Proof of Theorem \ref{thrm:cub_pdf}}\label{ap:cub_pdf_pf}
Without loss of generality, we can assume that $f'(\alpha) \geq a >0$. For brevity, denote $s(\alpha) = f_N ^{\rm spline} (\alpha)$, $h=h_{\max}$, and $\left\| f^{(m+1)} \right\|_{\infty} = \left\| f^{(m+1)} \right\|_{L^{\infty}[\alpha_{\min} , \alpha_{\max} ]}$. In general, $s(\alpha)$ can be non-monotone. By Theorem \ref{thm:hall_spl}, however, $\lvert s'(\alpha) -~f'(\alpha)\rvert <~C_{\rm spl}^{(1,m)} \left\| f^{(m+1)} \right\|_{\infty} h^m$. Hence

\begin{equation}\label{eq:s_mono}
s'(\alpha)\geq \frac{a}{2}>0 , \qquad N> \sqrt[m]{\frac{2C_{\rm spl}^{(1,m)} \left\| f^{(m+1)} \right\|_{\infty}}{a}}\left(\alpha_{\max} - \alpha_{\min} \right) \, ,
\end{equation}
and so $s(\alpha)$ is monotonically increasing and invertible for sufficiently large $N$.\footnote{In the numerical example \eqref{eq:tanh}, this lower bound is roughly $N>30$.} Because~$s(\alpha)$ interpolates $f(\alpha)$, and because both functions are monotone, then ${\rm range} \, (s) = {\rm range} \, (f)$. Since $s,f \in C^1$ and are invertible, by Lemma \ref{lem:pdf}
\begin{equation}\label{eq:l1_basic}
\|p_f-p_s\|_1 :\,= \int\limits_{f(\alpha_{\min})}^{f(\alpha_{\max})} \left| p_f(y)-p_s(y) \right | \, dy  =  \int\limits_{f(\alpha_{\min})}^{f(\alpha_{\max})} \left| \frac{c\left( f^{-1} (y)\right)}{f'\left(f^{-1}(y)\right)} - \frac{c\left( s^{-1} (y)\right)}{s'\left(s^{-1}(y)\right)} \right| \, dy \, .
\end{equation}
Denote $y=f(\alpha)$ and $\alpha_{\star} :\,= \alpha_{\star} (\alpha) = s^{-1} \left(f(\alpha)\right)$. Then by a change of variable
\begin{equation}\label{eq:l1_bd_middle}
\|p_f -p_s\|_1 =  \int\limits_{\alpha_{\min}}^{\alpha_{\max}}  \left| \frac{c(\alpha)}{f'(\alpha)} - \frac{c(\alpha_{\star})}{s'(\alpha_{\star})} \right| f'(\alpha)\, d\alpha =  \int\limits_{\alpha_{\min}}^{\alpha_{\max}}  \left| s'(\alpha_{\star})c(\alpha) - f'(\alpha)c(\alpha_{\star})  \right| \frac{1}{s'(\alpha_{\star})}\, d\alpha \, .
\end{equation}
For all $\alpha \in [\alpha_{\min}, \alpha_{\max}]$, 
$$
\lvert s'(\alpha_{\star})c(\alpha) - f'(\alpha)c(\alpha_{\star}) \rvert \leq  c(\alpha)\lvert s'(\alpha _
{\star})-s'(\alpha) \rvert + c(\alpha)\lvert s'(\alpha)-f'(\alpha) \rvert + f'(\alpha)\lvert c(\alpha)-c(\alpha_{\star}) \rvert\, .$$
Because $s'(\alpha)$ and $c(\alpha)$ are differentiable,
\begin{equation}\label{eq:numerator_bd}
\lvert s'(\alpha_{\star})c(\alpha) - f'(\alpha)c(\alpha_{\star}) \rvert \leq  D \lvert \alpha - \alpha_{\star} \rvert + 
\|c\|_{\infty} \lvert f'(\alpha) - s'(\alpha) \rvert \, \, ,
\end{equation}
where $D  = \left[ \| c \|_{\infty} \cdot \| s'' \|_{\infty} + \| c' \| _{\infty}  \cdot \|f\|_{\infty} \right]$.\footnote{By the same argument as \eqref{eq:s_mono}, for a fixed $\epsilon >0$ there exists a sufficiently large $N_0$ such that  $s''(\alpha) \leq~f''(\alpha) + \epsilon$ for all $N > N_0$. Therefore $\max \,  \|s'' \|_{\infty} \leq \max \,\| f''\|_{\infty} + \epsilon$, and so $D$ is independent of~$N$, and depends only on $f(\alpha)$, $c(\alpha)$, $\alpha_{\min}$ and $\alpha_{\max}$.} By Lagrange's mean-value theorem, there exists~$\beta$ between $\alpha$ and $\alpha_{\star}$ such that 
$$s(\alpha)-s(\alpha_{\star})  = s'(\beta)\left( \alpha - \alpha_{\star} \right) \, .$$
On the other hand, since~$\alpha_{\star}=s^{-1} (f(\alpha))$, then $s(\alpha_{\star}) = f(\alpha)$, and so
$$ s(\alpha) - s(\alpha_{\star})  = s(\alpha) -f(\alpha) \, .$$
Therefore $\alpha - \alpha_{\star} = \frac{s(\alpha)-f(\alpha)}{s'(\beta)}$. By \eqref{eq:s_mono}, $s'(\beta) \geq \frac{a}{2}$, and by Theorem~\ref{thm:hall_spl}, we have $\lvert f(\alpha)-s(\alpha) \rvert \leq C_{\rm spl}^{(0,m)} \left\|f^{(m+1)} \right\| _{\infty} h^{m+1}$. Hence, $$\lvert \alpha - \alpha_{\star} \rvert \leq \frac{2C_{\rm spl}^{(0,m)} \left\|f^{(m+1)} \right\| _{\infty}}{a}h^{m+1} .$$
By Theorem~\ref{thm:hall_spl}, $\lvert f'(\alpha)-s'(\alpha) \rvert \leq C_{\rm spl}^{(1,m)}\left\|f^{(m+1)} \right\| _{\infty} h^m$.
Hence \eqref{eq:numerator_bd} reads 
\begin{equation}\label{eq:numerator_bd_final}
\lvert s'(\alpha_{\star}) c(\alpha) - f'(\alpha) c(\alpha_{\star}) \rvert \leq K_1 h^m +K_2 h^{m+1} \, ,
\end{equation} 
where $K_1 = C_{\rm spl}^{(0,m)} \|c\|_{\infty} \left\|f^{(m+1)} \right\| _{\infty}$ and $K_2 = \frac{2}{a}C_{\rm spl}^{(0,m)} \left\|f^{(m+1)} \right\| _{\infty} D$. Substituting $\frac{1}{s'(\alpha_{\star})} \leq  \frac{2}{a}$, see \eqref{eq:s_mono}, and~\eqref{eq:numerator_bd_final} in \eqref{eq:l1_bd_middle}, for sufficiently large $N$ such that $h=\frac{\alpha_{\max}-\alpha_{\min}}{N-1}<1$ we have that $$\|p_f -p_s\|_1 \leq  \int\limits_{\alpha_{\min}}^{\alpha_{\max}} \frac{2(K_1 + K_2)}{a} h^m \, d\alpha  = \frac{2(K_1 + K_2)}{a} (\alpha_{\max} - \alpha_{\min}) h^m \leq \frac{K}{N^m}  \, ,$$
where $K=\frac{2(K_1 + K_2)}{a}\left( \alpha_{\max} - \alpha_{\min} \right)^{m+1}$.

{\revdone Similarly, by \eqref{eq:numerator_bd_final}, we have that for $1\leq q < \infty$, $$\|p_f-p_s\|_q ^q \leq \int\limits_{\alpha _{\min}} ^{\alpha _{\max}} \left| \frac{2(K_1 +K_2)}{a}h^m \right |^q \, d\alpha \leq K^q(q) h^{qm} \, , $$
for a suitable $K(q)>0$, and so $\|p_f-p_s\|_q \leq K(q) h^m \leq K(q)N^{-m}$.}

\remark If $f'(\alpha)=0$ for some values of $\alpha$, the approximation $p_s$ is not guaranteed to converge in the $L^1$ norm. By \eqref{eq:numerator_bd_final}, however, we can guarantee a third-order convergence for the pointwise error $p_f(y)-p_s(y)$, for every real number $y$ such that $f'(\alpha)$ does not vanish on~$\{\alpha \, \left| \right. \, f(\alpha)=~y \}$.

\section{Proof of Lemma \ref{lem:quar}}\label{sec:quar}

Similarly to the proof of \eqref{eq:numerator_bd},
$$\lvert s'(\alpha_{\star})c(\alpha) - f'(\alpha)c(\alpha_{\star}) \rvert \leq  D \lvert \alpha - \alpha_{\star} \rvert + 
c(\alpha) \lvert f'(\alpha) - s'(\alpha) \rvert \, \, .$$
Because $\lvert \alpha - \alpha_{\star}  \rvert  \leq  K_2 h^4$, then by \eqref{eq:l1_bd_middle},
\begin{equation}\label{eq:gammabd}
\| p_f -p_s\|_1 \leq \frac{2K_2}{a}h^4 +  \int\limits_{\alpha_{\min}} ^{\alpha_{\max}} \lvert f'(\alpha) - s'(\alpha) \rvert c(\alpha) \, d\alpha \, . 
\end{equation}
Since $f'(\alpha) - s'(\alpha)$ is continuous on $[\alpha_{\min},\alpha_{\max}]$, it vanishes and changes its sign only at~$J_N <~\infty$ points, denoted by $\alpha_{\min} = \gamma_0 < \gamma _1 < \cdots <\gamma _{J_N}= \alpha _{\max}$. Using integration by parts, the last integral reads 
$$\int\limits_{\alpha_{\min}} ^{\alpha_{\max}} \lvert f'(\alpha) - s'(\alpha) \rvert c(\alpha) \, d\alpha = \eta \sum\limits_{j=1}^{J -1} \left( -1 \right)^j \int\limits_{\gamma_{j}} ^{\gamma _{j+1}} \left( f'(\alpha) - s'(\alpha) \right) c(\alpha) \, d\alpha$$ $$ = \eta \sum\limits_{j=1}^{J_N -1} \left( -1 \right)^j\Big[ c(\gamma _{j+1})\left( f(\gamma_{j+1} ) - s(\gamma_{j+1}) \right) - c(\gamma _{j})\left( f(\gamma_{j} ) - s(\gamma_{j}) \right) - \int\limits_{\gamma_{j}} ^{\gamma _{j+1}} \left( f(\alpha) - s(\alpha) \right) c'(\alpha) \, d\alpha\Big] \, ,
$$
where $\eta = {\rm sign} \, \left[ f'(\alpha_{\min}) - s'(\alpha_{\min})\right]$. By Theorem \ref{thm:hall_spl}, ~$$\lvert c(\gamma _{j} ) (f(\gamma_j ) - s(\gamma _j) ) \rvert \leq \|c\|_{\infty} C_{\rm spl}^{(0)}\left\| f^{(4)} \right\|_{\infty}h^4  \, , \qquad 1\leq j \leq J_N \, ,  $$
and 
$$\Big| \int\limits_{\gamma_{j}} ^{\gamma _{j+1}} \left( f(\alpha) - s(\alpha) \right) c'(\alpha) \, d\alpha \Big| \leq  \left\| c' \right\|_{\infty}\left(\gamma _{j+1}  - \gamma_{j} \right) C_{\rm spl}^{(0)}\left\| f^{(4)} \right\|_{\infty}   h^4 \, , \qquad 1<j\leq J_N \,  .$$
Substituting these bounds in \eqref{eq:gammabd} yields $$ \| p_f -p_s\|_1 \leq \frac{2K_2}{a}h^4 +  K_3h^4+  K_4  J_N h^4 \, ,$$
where $K_3 = \|c'\|_{\infty} \lvert \alpha_{\max}  -\alpha_{\min}\rvert C_{\rm spl} ^{(0)}\left\| f^{(4)} \right\|_{\infty}$ and $K_4 =  2\|c\|_{\infty}C_{\rm spl} ^{(0)}\left\| f^{(4)} \right\|_{\infty}$. In the case of a uniform grid, the first two terms are $O(N^{-4})$, and the last term is $O(N^{-4}J_N)$, which completes the proof.

\section{Proof of Lemma \ref{lem:ddim_pdf}}\label{ap:ddim_pdf_pf}

For any $y\in \mathbb{R}$, the CDF of $f$ is

\begin{equation}\label{eq:multid_cdf}
P _f(y) =  {\rm Prob} \, \left\{ f(\pmb{\alpha}) \leq y\right\} = \frac{1}{\mu(\Omega)}\int_{D(y)}  \, d\mu (\pmb{\alpha})=  \frac{1}{\mu(\Omega)}\int_{D(y)} c(\pmb{\alpha}) \, d\pmb{\alpha} \, ,
\end{equation} where 
\begin{equation}\label{eq:DofY}D(y) :\, = \left\{ \pmb{\alpha}\in \Omega ~ | ~ f(\pmb{\alpha}) \leq y\right\} \, .
\end{equation}
To compute the PDF $p_f(y) :\,= \frac{d}{dy}P_f(y)$, we recall the {\em co-area formula}:
\begin{lemma}[\cite{evans2018measure}]\label{lem:coarea}
Let $A \subseteq \mathbb{R}^d$ be a Jordan set, let $u:A \to \mathbb{R}$ be Lipschitz and piecewise differentiable such that $u^{-1}(z)\subseteq A$ is a $(d-1)$ dimensional manifold for all $z\in \mathbb{R}$, and let~$g\in L^1 (A)$. Then
\begin{equation}\label{eq:coarea}
\int\limits_{A} g(\pmb{\alpha}) |\nabla u(\pmb{\alpha}) |  \, d\pmb{\alpha} = \int _{z\in u(A)} dz \int _{u^{-1} (z)} g(\pmb{\alpha}) \, d\sigma \, ,
\end{equation}
where $d\sigma$ is the $(d-1)$ dimensional surface element of $u^{-1} (z)$.
\end{lemma}
We apply the co-area formula to the right-hand-side of \eqref{eq:multid_cdf} by substituting $A= D(y)$, $g = \frac{c}{|\nabla f |}$ and $u=f$ in \eqref{eq:coarea}. The use of \eqref{eq:coarea} is justified since
\begin{enumerate}
\item $D(y)$ is bounded, since $\Omega$ is bounded. We can therefore show that $D(y)$ is Jordan by proving that $m(\partial D(y))=0$, where $m$ is the Lebesgue measure in $\mathbb{R}^d$. Since $\partial D(y) \subseteq f^{-1}(y) \cup \partial \Omega$, it is sufficient to show that each of these sets is of measure zero. Indeed, $\Omega$ is Jordan, and so $m(\partial \Omega) = 0$. In addition, since $|\nabla f| \neq 0$ on $f^{-1}(y)$, by the implicit function theorem $f^{-1}(y)$ is a $(d-1)$ dimensional manifold, and so $m(f^{-1}(y)) = 0$.

\item $f$ is piecewise-differentiable by the conditions of the Lemma \ref{lem:ddim_pdf}. Furthermore, because $f$ is piecewise-differentiable on a {\em compact set} $\bar{\Omega}$, it is also Lipschitz.

\item Since $f$ is continuous and $|\nabla f|\neq 0$ on $\bar{\Omega}$, then $\frac{1}{|\nabla f|}$ is bounded from above. Therefore, since $c\in L^1$, so is $g= \frac{c}{|\nabla f |}$.

\end{enumerate}
Thus, by Lemma \ref{lem:coarea} and \eqref{eq:multid_cdf},
\begin{equation}
P _f(y) =   \frac{1}{\mu(\Omega)}\int_{D(y)} c(\pmb{\alpha}) \, d\pmb{\alpha}  =  \frac{1}{\mu(\Omega)}\int_{-\infty} ^y dz \int _{f^{-1}(z)} \frac{c}{|\nabla f|} \, d\sigma \, .
\end{equation}
The outer integral on the right-hand-side is over $(-\infty , y)$ since $f(D(y))\subseteq (-\infty, y)$, see \eqref{eq:DofY}. Finally, since $p_f(y) = \frac{d}{dy} P_f(y)$, differentiating the last integral using the (one-dimensional) Leibnitz integral rule yields \eqref{eq:multid_pdf}.

\section{Proof of Theorem \ref{thm:multid}}\label{app:multid_pdf_pf}

Since $f\in C^{m+1}(\Omega)$ and $\Omega$ is compact, $f$ is also Lipschitz. Hence, Lemma \ref{lem:ddim_pdf} can be applied with $m([0,1]^d) = 1$ and $c(\pmb{\alpha})\equiv 1$, yielding
\begin{equation}\label{eq:Iy_def2}
\| p_f -p_s\|_q ^q = \int\limits_{-\infty}^{\infty}  I^q(y) \, dy , \qquad I(y):\,= \Big|  \int_{f^{-1} (y)} \frac{1}{|\nabla f |} \,  d\sigma -   \int_{s^{-1} (y)} \frac{1}{|\nabla s |} \,  d\sigma \, \Big| \, ,
\end{equation}
where $\sigma$ is the $d-1$-dimensional surface measures induced by the Lebesgue measure. 

The outline of the proof is as follows:
\begin{enumerate}
\item For a fixed $y$ in the image of $s(\pmb{\alpha})$, we construct a cover  $\left\{A_j (y) \right\} _{j=1} ^d$ of $s^{-1}(y)$.

\item We then construct a set of maps $\phi _j :A_j (y) \to f^{-1} (y)$, which are characterized in Lemma \ref{lem:phij2}.  
\item We construct a {\em disjoint} cover $\tilde{A}_j \subseteq A_j (y)$ for $1\leq j \leq d$. Lemma \ref{lem:phijBj_almost_cover} proves that $\big\{\phi _j (\tilde{A}_j ) \big\} _{j=1} ^d$ are mutually disjoint, up to an $O(h^m)$ error, and  almost cover $f^{-1} (y)$, up to an $O(h^m)$ error.
\item By an inclusion-exclusion argument and the implicit function theorem, we split the integral of $I(y)$ to $d$ integrals over compact domains in $\mathbb{R}^{d-1}$.
\item By Theorem \ref{thm:tensor_spline}, and similarly to the proof of the one-dimensional counterpart (Theorem \ref{thrm:cub_pdf}), we bound each of the integrals obtained in step 3. Thus, we obtain a {\em pointwise} bound on $p_f(y)-p_s(y)$.
\item Finally, we use compactness of $\Omega$ and the fact that $f,s\in C^1(\Omega)$ to bound $\| p- p_s\|_1$.
\end{enumerate}

\subsection*{Step 1}
For brevity, denote by $\partial _{\alpha _j} = \frac{\partial}{\partial _{\alpha _j}}$ the partial derivative along the $j$-th axis for $1\leq j \leq d$. Fix $y$, and let $A_j=A_j (y) \subseteq s^{-1}(y)$ be defined by
\begin{equation}\label{eq:Aj_def2}
A_j :\,= \left\{ \pmb{\alpha} \in s^{-1} (y) ~ \Big| ~ \left|  \partial \alpha _j f(\pmb{\alpha}) \right| > \frac{\kappa _{\rm f}}{d} \right\} \,  \qquad j=1,\ldots, d \, .
\end{equation}
Since $|\nabla f| = \sqrt{\sum _{j=1} ^d (\partial _{\alpha_j} f )^2}\geq~\kappa _{\rm f}$ on $\Omega$, for every $\pmb{\alpha} \in s^{-1} (y)$ {\em at least} one component of $\nabla f (\pmb{\alpha}) $ satisfies~$|\partial _{\alpha _j} f| \geq \frac{\kappa _{\rm f}}{d}$.\footnote{\label{footnote_cover} Since $\kappa _{\rm f} \leq \sqrt{\sum_{j=1}^d (\partial_{\alpha_j}f)^2} \leq \sqrt{d}\max\limits_{j=1,\ldots ,d}|\partial _{\alpha _j} f|$, then $\max\limits_{j=1,\ldots ,d}|\partial _{\alpha _j} f| \geq \frac{\kappa _{\rm f}}{\sqrt{d}}> \frac{\kappa _{\rm f}}{d}$.} Hence, $\pmb{\alpha} \in A_j (y)$ for some $1\leq j\leq d$, and so
\begin{equation}\label{eq:Aj_cover2}
s^{-1}(y) = \cup _{j=1} ^d A_j (y)  \, .
\end{equation} 
\subsection*{Step 2}
Next, we prove the existence of the maps $\phi _j:A_j  \to f^{-1} (y)$.
\begin{lemma}\label{lem:phij2}

Let $\pmb{\alpha} \in A_j(y)$ and let $h$ be defined as in Theorem \ref{thm:tensor_spline}. Then for a sufficiently small~$h>0$,~there exists a real number $\delta = \delta (\pmb{\alpha})$ such that 
\begin{enumerate}

\item $\pmb{\alpha} + \delta(\pmb{\alpha} )\hat{e}_j  \, \in f^{-1} (y)$,
where $\hat{e} _j $ is the unit vector in the direction of the $j$-th axis.


\item The maps 
\begin{equation}\label{eq:phi_def2}
\phi _j(\pmb{\alpha}) :\,= \pmb{\alpha} + \delta (\pmb{\alpha}) \hat{e} _j \,   , \qquad j=1,\ldots, d \, .
\end{equation}
are injective from $A_j = A_j (y)$ to $f^{-1} (y)$.
\item For every $\pmb{\alpha}\in A_j$,
\begin{equation}\label{eq:phidelta_bd}
\delta (\pmb{\alpha}) = O(h^{m+1}) \, .
\end{equation}

\item For every~$E\subseteq A_j$, 
\begin{equation}\label{eq:measure_preserve}
\left|\sigma (E)-\sigma (\phi_j (E))\right| = O(h^{m}) \, ,
\end{equation}
where as in \eqref{eq:Iy_def2}, $\sigma$ is the $d-1$ dimensional surface measure induced by the Lebesgue measure on $\Omega$.
\end{enumerate}
\end{lemma}
\begin{proof}
\begin{enumerate}
\item We prove this for the case where $y> f(\pmb{\alpha})$ and $\partial _{\alpha _j} f (\pmb{\alpha}) > 0$ on $\Omega$. The proofs for the three other cases are similar. Since $f\in C^{m+1} (\Omega )$ and $\Omega$ is compact, all the second derivatives of $f$ are bounded, and so $|\partial _{\alpha _j } ^2 f | <~M_2<\infty$ on $\Omega$. Hence, since $\partial_{\alpha _j} f(\pmb{\alpha})>\frac{\kappa _{\rm f}}{d}$, there exists a segment $L = L(\pmb{\alpha}) =  \big\{ \pmb{\alpha } + \xi \hat{e} _j ~ , ~ |\xi| < \xi _{\max} \big\}$, where $\xi _{\max}$ depends {\em only} on~$M_2$, such that $\partial _{\alpha _j} f > \frac{\kappa _{\rm f}}{2d}$ on $L$. Therefore $f(\pmb{\alpha} + \xi_{\max} \hat{e} _j)> f(\pmb{\alpha})+ \frac{\kappa _{\rm f}}{2d}\xi _{\max}$. By the mean-value theorem, $f$ attains on $L$ all values in $[f(\pmb{\alpha}),  f(\pmb{\alpha})+ \frac{\kappa _{\rm f}}{2d}\xi _{\max}]$. 

Now, by Theorem \ref{thm:tensor_spline}, since $\pmb{\alpha} \in s^{-1} (y)$ and since $y>f(\pmb{\alpha})$,
\begin{equation}\label{eq:yalpha_hm}
 y -f(\pmb{\alpha}) = s(\pmb{\alpha})-f(\pmb{\alpha}) \leq C_m h^{m+1} \, . 
\end{equation}
Hence, for $h$ sufficiently small, $y \in [f(\pmb{\alpha}),  f(\pmb{\alpha})+ \frac{\kappa _{\rm f}\xi _{\max}}{2d}]$, and so there exists a point  $\pmb{\alpha} + \delta(\pmb{\alpha})\hat{e}_j \, \in L$ such that $ f(\pmb{\alpha} + \delta  (\pmb{\alpha}) \hat{e}_j) = y$. 

\item Assume by negation that $\phi _j$ is {\em not} injective. Then there exist $\pmb{\alpha }^1, \pmb{\alpha }^2 \in A_j $ such that $\phi _j (\pmb{\alpha }^1)= \phi_j ( \pmb{\alpha }^2) =~\pmb{\lambda}$. Since $\phi _j$ only changes the $j$-th coordinate, see \eqref{eq:phi_def2}, we can regard $s$ and $f$ as single-variable functions of the $j$-th coordinate $\alpha _j$. Since $\phi _j (\pmb{\alpha }^1)= \phi_j ( \pmb{\alpha}^2)=\pmb{\lambda}$, from the proof of item (1) in this lemma it follows that $\pmb{\lambda} \in L(\pmb{\alpha}^1)\cap L(\pmb{\alpha}^2)$. Hence, the segment between $\pmb{\alpha }^1$ and $\pmb{\alpha}^2$ is contained in $L(\pmb{\alpha}^1)\cup L(\pmb{\alpha}^2)$, where we know that $|\partial _{\alpha _j} f|>\frac{\kappa _{\rm f}}{2d}$. By Theorem \ref{thm:tensor_spline}, this means that if $h$ is sufficiently small, $|\partial _{\alpha _j} s|>0$ on the segment between $\pmb{\alpha}^1$ and $\pmb{\alpha}^2$. This leads to a contradiction, since on the one hand $\pmb{\alpha} ^1, \pmb{\alpha} ^2 \in A_j (y) \subseteq s^{-1}(y)$, and so $s(\pmb{\alpha} ^1) = s(\pmb{\alpha }^2) = y$, but on the other hand $s(\pmb{\alpha})$ is strictly monotone on the segment between $\pmb{\alpha}^1$ and $\pmb{\alpha}^2$.

\item Since $f\in C^2$, and by \eqref{eq:phi_def2},
\begin{equation}\label{eq:phij_alphaj_pre}
\partial _{\alpha_j} f(\phi _j (\pmb{\alpha}))- \partial _{\alpha _j}f(\pmb{\alpha})=\partial _{\alpha_j} f(\pmb{\alpha}+\delta (\pmb{\alpha})\hat{e}_j)- \partial _{\alpha _j}f(\pmb{\alpha})= O(\delta(\pmb{\alpha})) \, .
\end{equation}
In addition, by Lagrange mean-value theorem, for any $\pmb{\alpha}\in s^{-1}(y)$ $$s(\pmb{\alpha}) -f(\pmb{\alpha}) = y-f(\pmb{\alpha}) = f(\pmb{\alpha} +\delta(\pmb{\alpha}) \hat{e}_j) -f(\pmb{\alpha}) = \partial _{\alpha _j} f(\pmb{\alpha}+\zeta \hat{e}_j) \cdot \delta(\pmb{\alpha}) \,, \qquad 0\leq \zeta \leq \delta \, .$$ Hence, using Theorem \ref{thm:tensor_spline}, and since $|\partial _{\alpha _j} f|\geq \frac{\kappa _{\rm f}}{2d}$ on the segment between $\pmb{\alpha}$ and $\phi _j (\pmb{\alpha})$ (see proof of item 1 in this lemma), we have that
\begin{equation}
|\delta (\pmb{\alpha})| = \left|\frac{s(\pmb{\alpha}) - f(\pmb{\alpha})}{\partial _{\alpha _j} f(\pmb{\alpha}+\zeta \hat{e}_j)} \right| \leq \frac{C_m h^{m+1}}{\frac{\kappa _{\rm f}}{2d}} = O(h^{m+1}) \, .
\end{equation}
\item For brevity of notations and without loss of generality, fix $j=d$, and let $E\subseteq A_d$. In this case, $\partial _{\alpha _d} s \neq 0$ on $E$,\footnote{$\partial _{\alpha _d} s \neq 0$ on $A_d$ for sufficiently small $h$ since $|\partial _{\alpha _d}f|\geq \frac{\kappa _{\rm f}}{d}$ on $A_d$, and since by Theorem~\ref{thm:tensor_spline} $|\partial _{\alpha _d} s - \partial _{\alpha _d} d| =O(h^m)$.} and so by the implicit function theorem there exists a function~$S$,~such that if $s(\alpha_1, \ldots, \alpha_d)=y$, then $\alpha _d = S(\alpha_1 , \ldots , \alpha _{d-1})$. The domain of $S$ is $$G_E :\,=  \big\{ (\alpha_1 , \ldots ,\alpha_{d-1} )  ~ | ~ \exists \alpha _d \in [0,1] ~~{\rm s.t.}~~ (\alpha_1, \ldots ,\alpha_d) \in E \big\}  \, .$$
In particular, if $(\alpha _1, \ldots ,\alpha_{d-1})\in G_E$, then $s(\alpha _1, \ldots , \alpha _{d-1}, S(\alpha _1, \ldots ,\alpha _{d-1})) =~y$. Therefore $$\sigma (E)= \int_{E} 1 \, d\sigma  = \int_{G_E}  \sqrt{1+|\nabla S|^2} \, d\alpha _1 \cdots d\alpha _{d-1} \, . $$
Furthermore, by the implicit function theorem, $\partial _{\alpha _j} S = -\frac{\partial _{\alpha _j}s}{\partial _{\alpha _d}s}$ for $1\leq j <d$, and so $$\sqrt{1 + |\nabla S|^2} = \sqrt{1 + \sum _{j=1}^{d-1} \left(\frac{\partial _{\alpha _j}s}{\partial _{\alpha _d}s}\right)^2}  = \frac{1}{\left|\partial _{\alpha _d}s \right|}\sqrt{\left(\partial _{\alpha _d}s\right)^2 + \sum _{j=1}^{d-1} \left(\partial _{\alpha _j}s\right)^2} = \frac{1}{\left|\partial _{\alpha _d}s \right|}|\nabla s| \, .$$ Hence, 
\begin{subequations}\label{eq:almost_measure_pre}
\begin{equation}
\sigma(E) = \int_{G_{E}} \frac{|\nabla s|}{\left|\partial _{\alpha _d}s \right|} \,  d\alpha _1 \cdots d\alpha _{d-1} \, .
\end{equation}
Next, since $|\partial _{\alpha _d} f|\geq \frac{\kappa _{\rm f}}{2d}$ on $\phi _d (E)$ (see proof of item 1 in this lemma), we similarly apply the implicit function on $\phi _d (E)$: there exists function $F:G_{\phi_d (E)} \to \mathbb{R}$ where $G_{\phi _d (E)} \subset \mathbb{R}^{d-1}$, such that $f(\alpha _1 , \ldots , \alpha _{d-1}, F(\alpha _1 , \ldots , \alpha _{d-1}))=~y$. Hence, since $\phi _d (E) \subseteq f^{-1} (y)$,
\begin{equation}
\sigma(\phi _d (E)) = \int_{\phi _d (E)}1 \, d\sigma = \int_{G_{\phi _d (E)}} \frac{|\nabla f|}{\left|\partial _{\alpha _d}f\right|} \,  d\alpha _1 \cdots d\alpha _{d-1} \, .
\end{equation}
\end{subequations}
Next, by item 2 of this lemma, then $\phi _d$ induces a bijection $\varphi_d: G_{E} \to G_{\phi _d (E)}$. But, because $\phi _d$ only alters the $\alpha _d$ coordinate, $\varphi_d =  {\rm Id}$, and so $ G_{E} = G_{\phi _d (E)}$. Using this equality and~\eqref{eq:almost_measure_pre}~yields
\begin{equation}\label{eq:almost_measure_premitive}
\begin{aligned}
\big| \sigma (E) - \sigma &(\phi _d (E)) \big| = \left| \int_{G_{E}}   \left( \frac{|\nabla f(\phi _d (\pmb{\beta}))|}{\left|\partial _{\alpha _d}f(\phi _d (\pmb{\beta})) \right|} - \frac{|\nabla s(\pmb{\beta})|}{\left|\partial _{\alpha _d}s(\pmb{\beta})\right|}  \right) \,  d\alpha _1 \cdots d\alpha _{d-1}\right| \\
&\leq \int_{G_{E}}  \frac{\big| |\nabla f(\phi _d (\pmb{\beta}))| \cdot | \partial _{\alpha _d}s(\pmb{\beta})| - |\nabla s(\pmb{\beta})|\cdot |\partial _{\alpha _d}f(\phi _d (\pmb{\beta})) | \big|}{| \partial _{\alpha _d}f(\phi _d (\pmb{\beta}))|\cdot | \partial _{\alpha _d}s(\pmb{\beta})|}  \,  d\alpha _1 \cdots d\alpha _{d-1} \, ,
\end{aligned}
\end{equation}

where for brevity, we denote $\pmb{\beta } :\,= (\alpha _1, \ldots , \alpha _{d-1}, S(\alpha _1, \ldots , \alpha _{d-1}))\in E$ and note that by \eqref{eq:phi_def2} $$(\alpha _1, \ldots , \alpha _{d-1}, F(\alpha _1, \ldots , \alpha _{d-1})) = \phi _d (\pmb{\beta})  \, .$$

To bound the right-hand-side of \eqref{eq:almost_measure_premitive},  note that since $|\partial _{\alpha _d} f|>\frac{\kappa _{\rm f}}{d}$ on $E$, and since by Theorem \ref{thm:tensor_spline}, $\big| \partial _{\alpha _d}s - \partial _{\alpha _d}f \big| \leq C_m h^m$, then for a sufficiently small $h$, $\big| \partial _{\alpha _d}s\big| > \frac{\kappa _{\rm f}}{2d}$ on $E$. Substituting these bounds in~\eqref{eq:almost_measure_premitive} yields
\begin{equation}\label{eq:measure_preserve_3}
\begin{aligned}
\big| \sigma (E) - &\sigma (\phi _d (E)) \big| \leq \\
&\frac{2d^2}{\kappa _{\rm f}^2}  \int_{G_{E}}\big| |\nabla f (\phi _d (\pmb{\beta}))| \cdot |\partial _{\alpha _d} s(\pmb{\beta})|
-  |\nabla s (\pmb{\beta})| \cdot |\partial _{\alpha _d} f(\phi _d (\pmb{\beta}))| \big|    d\alpha _1 \cdots d\alpha _{d-1}   \, .
\end{aligned}
\end{equation}

Therefore, we can rewrite and bound the right-hand-side integrand by
\begin{equation}\label{eq:integ_measure_triang}
\begin{aligned}
\big| |\nabla f &(\phi _d (\pmb{\beta}))| \cdot |\partial _{\alpha _d} s(\pmb{\beta})|
-  |\nabla s (\pmb{\beta})| \cdot |\partial _{\alpha _d} f(\phi _d (\pmb{\beta}))| \big|  \leq \\ 
& |\partial _{\alpha _d} s(\pmb{\beta})| \cdot \big| |\nabla f (\phi _d (\pmb{\beta}))| -|\nabla f ( \pmb{\beta})| \big| 
\quad +\quad 
|\nabla f ( \pmb{\beta})| \cdot \big| |\partial _{\alpha _d} s(\pmb{\beta})| - |\partial _{\alpha _d} s(\phi _d (\pmb{\beta}))| \big| \\
+& |\nabla f (\pmb{\beta})| \cdot \big| |\partial _{\alpha _d} s(\phi _d (\pmb{\beta}))|- |\partial _{\alpha _d} f(\phi _d (\pmb{\beta}))| \big| 
\quad +\quad 
|\partial _{\alpha _d} f(\phi _d (\pmb{\beta}))| \cdot \big ||\nabla f (\pmb{\beta})| - |\nabla s (\pmb{\beta})|\big| \, .
\end{aligned}
\end{equation}

Since $s,f\in C^2(\Omega)$ and $\Omega$ is compact, $\partial _{\alpha _d} s$, $\partial _{\alpha _d} f$ and $\nabla f$ are bounded on $\Omega$. Furthermore, since $s,f\in C^2$, the first and second term in the right-hand-side of \eqref{eq:integ_measure_triang} are $O(\delta)$, and so by \eqref{eq:phidelta_bd} both of these terms are $O(h^{m+1})$. In addition, by Theorem \ref{thm:tensor_spline} the third and fourth term on the right-hand-side of \eqref{eq:integ_measure_triang} are $O(h^m)$. Hence, the left-hand-side of \eqref{eq:integ_measure_triang} is~$O(h^m)$,~and so  finally, \eqref{eq:measure_preserve_3} reads 
$$\big| \sigma (E) - \sigma (\phi_d (E)) \big| \leq \frac{2d^2}{\kappa _{\rm f} ^2} \int _{G_E} Kh^m \, d\alpha _1 , \ldots , d_{\alpha _{d-1}} \leq \tilde{K} h^m \, ,$$
for some constant $\tilde{K} >0$.
\end{enumerate}
\end{proof}

We finish this step by noting that Lemma \ref{lem:phij2} would still hold if we interchange $f$ and $s$. Hence,
\begin{corollary}\label{corr:phij_reverse}
There exists sets $B_j \subseteq f^{-1} (y)$ such that $f^{-1} (y) = \cup_{j=1}^d B_j$ and maps $\tilde{\phi} _j :B_j \to s^{-1}(y)$ for which items 1-4 of Lemma \ref{lem:phij2} holds, interchanging $f$ and $s$.
\end{corollary}

\subsection*{Step 3}
Next, we re-partition $s^{-1}(y)$ into {\em disjoint} sets $\left\{ \tilde{A}_j \right\} _{j=1}^d$ where $\tilde{A}_j \subseteq A_j$ for every $1\leq j \leq d$. Let $\tilde{A}_1 :\, = A_1$, and define
\begin{equation}
\tilde{A} _j :\,= A_j \setminus \left( \cup _{k=1}^{j-1} \tilde{A}_k\right) \, , \qquad 1<j\leq d \, .
\end{equation}
Since by construction, $\cup_{j=1}^d \tilde{A} _j = \cup _{j=1}^d A_j$ , and since by \eqref{eq:Aj_cover2} $\cup _{j=1}^{d} A_j = s^{-1} (y) $, then $$\cup _{j=1}^d \tilde{A} _j =s^{-1}(y) \, .$$
Hence, since the sets $\big\{\tilde{A} _j\big\} _{j=1}^{d}$ are disjoint, we can rewrite the first component of $I(y)$, see \eqref{eq:Iy_def2}, as
\begin{equation}\label{eq:Iy_disjoint_Bj}
\int_{s^{-1}(y)} \frac{1}{|\nabla s |} \, d\sigma =\sum\limits_{j=1}^d  \int_{\tilde{A} _j} \frac{1}{|\nabla s |} \, d\sigma \, .
\end{equation}
To prove a counterpart of \eqref{eq:Iy_disjoint_Bj} for $\int _{f^{-1} (y)} \frac{1}{|\nabla f|}1 \, d\sigma$, we first prove the following Lemma:
\begin{lemma}\label{lem:phijBj_almost_cover}
Let $\sigma $ be the surface measure on $f^{-1} (y)$, let $\big\{ \tilde{A} _j\big \} _{j=1}^d$ be defined by \eqref{eq:Iy_disjoint_Bj} and~$\left\{ \phi _j \right\} _{j=1}^{d}$ be defined by \eqref{eq:phi_def2}. 
\begin{enumerate}
\item For any $1\leq k,j \leq d$ with $k\neq j$, then
\begin{equation}\label{eq:phij_almost_disjoint}
 \sigma \left( \phi _j(\tilde{A} _j) \cap \phi _k (\tilde{A}_k)\right) = O(h^{m}) \, . 
\end{equation}

\item \begin{equation}\label{eq:phij_almost_surj}
 \sigma \left(f^{-1}(y)  \setminus \cup _{j=1}^{d} \phi _j (\tilde{A} _j) \right) = O(h^{m}) \, . 
\end{equation}

\end{enumerate}
\end{lemma}
\begin{proof}
\begin{enumerate}
\item Fix the indices $j\neq k$ and denote for brevity $D_{jk} = \phi _j (\tilde{A} _j) \cap \phi _k (\tilde{A}_k)$. Let $\pmb{\beta} \in D_{jk}$. By injectivity of $\phi _j$ and $\phi _k$ (see Lemma \ref{lem:phij2}), There exist unique points $\pmb{\alpha} ^{(j)} \in \tilde{A} _j$ and $\pmb{\alpha} ^{(k)} \in \tilde{A}_k$ such that $\phi _j(\pmb{\alpha} ^{(j)}) = \phi _k(\pmb{\alpha} ^{(k)}) = \pmb{\beta}$. By definition \eqref{eq:phi_def2}, $$\pmb{\beta}-\pmb{\alpha} ^{(j)} = \delta(\pmb{\alpha} ^{(j)})\hat{e}_j \, , \qquad \pmb{\beta}-\pmb{\alpha} ^{(k)} = \delta(\pmb{\alpha} ^{(k)})\hat{e}_k \, .$$ Since $\hat{e}_j \perp \hat{e}_k$ and since by \eqref{eq:phidelta_bd} $\delta (\pmb{\alpha}^{j}),\delta (\pmb{\alpha}^{j})= O(h^{m+1})$, then\footnote{Geometrically, the points $\pmb{\alpha}^{(j)}$,$\pmb{\alpha}^{(k)}$ and $\pmb{\beta}$ are the vertices of a right-angle triangle, where both legs are $O(h^{m+1})$. Hence, by the Pythagorean Theorem, the length of the hypotenuse is also $O(h^{m+1})$.} $$|\pmb{\alpha} ^{(j)} - \pmb{\alpha} ^{(k)}| = O(h^{m+1}) \, .$$
Next, denote the {\em geodesic distance} on $s^{-1}$ by $|\cdot |_s$. Since $s\in C^1$, then $|\nabla s|$ is bounded from above on $\Omega$ and so $|\pmb{\alpha} ^{(j)}-\pmb{\alpha} ^{(k)}|_s =O(h^{m+1})$ as well. But since the interiors of $\tilde{A} _j$ and $ \tilde{A}_k$ are disjoint, then the geodesic path between $\pmb{\alpha }^{(j)}$ and $\pmb{\alpha} ^{(k)}$ must pass through a point $\pmb{\alpha} ^{\star} \in \partial \tilde{A} _j \cap \partial \tilde{A}_k$. Hence, 
\begin{equation}\label{eq:border_dist}
|\pmb{\alpha}^{\star} - \pmb{\alpha } ^{(j)} |_s = O(h^{m+1}) \, ,
\end{equation}

Since \eqref{eq:border_dist} holds for any $\pmb{\beta} \in D_{jk}$ and $\pmb{\alpha} ^{(j)} = \phi _j^{-1} (\pmb{\beta})$, then
$$\phi _j ^{-1} (D_{jk}) \subseteq E_{jk}(h) :\, = \left\{\pmb{\alpha} \in s^{-1}(y) ~~ |~ ~ \inf_{\pmb{\alpha}^{\star} \in \partial \tilde{A} _j \cap \partial \tilde{A}_k}|\pmb{\alpha}-\pmb{\alpha} ^{\star}|_s \leq Kh^{m+1} \right\} \, ,$$ for some $K>0$. It is therefore sufficient to show that $\sigma (E_{jk} (h)) = O(h^m)$ for $0< h\ll 0$.

By construction, $\partial \tilde{A} _j \cap \partial \tilde{A}_k\subseteq \cup _{j=1}^d \partial A_j$. Since $f\in C^1$, then $\sigma (\cup _{j=1}^d \partial A_j) = 0$ and so by monotonicity of measure $ \sigma(\partial \tilde{A} _j \cap \partial \tilde{A}_k)  =0$ as well.\footnote{For each $1\leq j \leq d$, the set $\partial A_j$ is the boundary of the smooth manifold $A_j$, and so it is of measure zero.} Furthermore $\partial \tilde{A} _j \cap \partial \tilde{A}_k$, is a finite union of smooth subsurface of $s^{-1}(y)$, each of finite ($d-2$)-dimensional surface measure.\footnote{For example, if $d=3$, than $\partial \tilde{A} _j \cap \partial \tilde{A}_k$ is a finite set of curves, each with a finite length.}  Finally, since $\partial \tilde{A} _j \cap \partial \tilde{A}_k$ is compact in the topology of the smooth $(d-1)$-dimensional manifold $s^{-1} (y)$ (it is bounded and close), and since $E_{jk} (h)$ is of {\em geodesic} radius $Kh^{m+1}$ from $\partial \tilde{A} _j \cap \partial \tilde{A}_k$, then $\sigma (E_{jk}) = O\big((h^{(m+1)})^{(d-1)}\big)\leq O(h^m)$. Hence,
\begin{equation}\label{eq:measure_cap_preimage}
\sigma \big( \phi _j ^{-1} \left( D_{jk} \right)  \big) \leq \sigma\big( E_{jk} (h)\big)=  O(h^m) \, .
\end{equation}

In addition, since $\phi _j$ is injective, $\phi _j (\phi _j ^{-1} (D_{jk})) = D_{jk}$. Hence, by taking $E = \phi _j ^{-1} (D_{jk})$ in \eqref{eq:measure_preserve} yields $$|\sigma (\phi _j ^{-1} (D_{jk}) ) - \sigma  ( D_{jk}) | = |\sigma (E ) - \sigma  (\phi _j (E)) | \leq O(h^m) \, .$$
Combined with \eqref{eq:measure_cap_preimage} this proves that $\sigma  ( D_{jk})  = O(h^m)$, as required.

\item Since $\cup _{j=1}^d \phi _j (\tilde{A} _j) \subseteq f^{-1}(y)$, then
\begin{subequations}\label{eq:almost_surj_oneside}
\begin{equation}\label{eq:f_phi_mono}
\sigma \left(\cup _{j=1}^d \phi _j (\tilde{A} _j)\right) \leq \sigma \left(f^{-1}(y)\right) \, .
\end{equation}
On the other hand, by item \eqref{eq:phij_almost_disjoint} and by the inclusion-exclusion argument
\begin{align*}
\sigma \left( \cup_{j=1}^d \phi _j (\tilde{A} _j)\right) &= \\
\sum\limits_{j=1}^d \sigma \left( \phi _j (\tilde{A} _j ) \right) - &\sum\limits_{j_1,j_2} \sigma \left( \phi _{j_1} (\tilde{A}_{j_1} ) \cap \phi _{j_2} (\tilde{A}_{j_2} ) \right) + \cdots + (-1)^{d+1} \sigma \left( \phi _{1} (\tilde{A}_{1} ) \cap \cdots \cap \phi _{d} (\tilde{A}_{d} ) \right) = \\
&\qquad \qquad \qquad \qquad \sum\limits_{j=1}^d \sigma \left( \phi _j (\tilde{A} _j ) \right)+ O(h^m) = \sum\limits_{j=1}^d \sigma \big( \tilde{A} _j  \big)+ O(h^m)\, .
\end{align*}
where the last equality is due to \eqref{eq:measure_preserve}. Hence,
\begin{equation}\label{eq:ineq_phijimag_s}
\sigma \left( \cup_{j=1}^d \phi _j (\tilde{A} _j)\right) = \sum\limits_{j=1}^d \sigma \big( \tilde{A} _j  \big)+ O(h^m) = \sigma (\cup _{j=1} ^d \tilde{A} _j ) + O(h^m) = \sigma (s^{-1}(y)) + O(h^m)\, ,
\end{equation}
\end{subequations}
where the second equality follows from the fact that the sets $\big\{ \tilde{A} _j \big\}_{j=1}^d$ are disjoint, and the third equality follows from $\cup _{j=1} ^d \tilde{A} _j = s^{-1} (y)$.

Since the left-hand-sides of \eqref{eq:f_phi_mono} and \eqref{eq:ineq_phijimag_s} are identical, it follows that
\begin{subequations}\label{eq:almost_surj_ineq}
\begin{equation}
\sigma \left( s^{-1} (y) \right) + O(h^m) \leq \sigma \left(f^{-1} (y)\right) \, .
\end{equation}
Crucially, since by Corollary \ref{corr:phij_reverse}, both Lemma \ref{lem:phij2} and item 1 of this lemma remain valid if we interchange $f$ and $s$, we also have that 
\begin{equation}
\sigma  \left( f^{-1} (y) \right) + O(h^m) \leq \sigma \left(s^{-1} (y)\right) \, .
\end{equation}
\end{subequations}
Combining the two inequalities of \eqref{eq:almost_surj_ineq} yields that
\begin{equation}\label{eq:fs_similar_sigma}
|\sigma (f^{-1} (y)) - \sigma (s^{-1}(y))| = O(h^m) \, .
\end{equation}
Finally
\begin{align*}
 \sigma \left( f^{-1} (y) \setminus \cup _{j=1}^d \phi _j (\tilde{A} _j) \right) &= \sigma\left( f^{-1} (y) \right)- \sigma\left( \cup _{j=1}^d \phi _j (\tilde{A} _j) \right) \leq  \\
&\left| \sigma\left( f^{-1} (y) \right)-\sigma\left( s^{-1} (y) \right)\right| + O(h^m) = O(h^m) \, ,
\end{align*}
where the inequality in the first line is due to \eqref{eq:ineq_phijimag_s}, and the last equality is due to \eqref{eq:fs_similar_sigma}.
\end{enumerate}
\end{proof}

\subsection*{Step 4}
By \eqref{eq:phij_almost_surj}, and since $\frac{1}{|\nabla f|} \leq \frac{1}{\kappa _{\rm f}}$, then $$\int_{f^{-1} (y)}\frac{1}{|\nabla f|} \, d\sigma= \int_{\cup_{j=1}^{d} \phi _j (\tilde{A} _j)}\frac{1}{|\nabla f|} \, d\sigma+ O(h^{m}) \, .$$
Hence, by an inclusion-exclusion argument,
\begin{equation}
\begin{aligned}
\int_{f^{-1} (y)} \frac{1}{|\nabla f|} \,  d\sigma &= O(h^{m})\, + 
\sum\limits_{j=1}^d  \int_{\phi_j(\tilde{A} _j) } \frac{1}{|\nabla f|} \,  d\sigma \, \\
&- \sum\limits_{\substack{j_1 <j_2 \\ j_1 =1}}^d   \int_{\phi _{j_1}( \tilde{A}_{j_1}) \cap \phi _{j_2} (\tilde{A}_{j_2})} \frac{1}{|\nabla f|} \,  d\sigma+ \cdots + (-1)^{d-1}  \int_{\phi _1 (\tilde{A}_{1})\cap \cdots \cap \phi _d (\tilde{A}_d)} \frac{1}{|\nabla f|} \,  d\sigma\, .
\end{aligned}
\end{equation}
But, by \eqref{eq:phij_almost_disjoint}, we can reduce all of the higher-order terms to yield
\begin{equation}\label{eq:Iy_disjoint_phijBj}
\int_{f^{-1}(y)} \frac{1}{|\nabla s |} \, d\sigma =\sum\limits_{j=1}^d  \int_{\phi _j (\tilde{A} _j)} \frac{1}{|\nabla f |} \, d\sigma + O(h^{m})\, .
\end{equation}
Hence, substituting \eqref{eq:Iy_disjoint_Bj} and \eqref{eq:Iy_disjoint_phijBj} into \eqref{eq:Iy_def2} yields
\begin{equation}\label{eq:Iy_breakdown}
I(y) \leq  \sum\limits_{j=1}^{d} \left|\int_{\phi _j (\tilde{A} _j)} \frac{1}{|\nabla f |} \, d\sigma - \int_{ \tilde{A} _j} \frac{1}{|\nabla s |} \, d\sigma  \right| + O(h^{m}) \, .
\end{equation}

\subsection*{Step 5}
By \eqref{eq:Iy_breakdown}, in order to show that $I(y) = O(h^m)$, it is sufficient to prove that 
\begin{equation}\label{eq:Ijy_def}
I_j(y) :\,= \left|\int_{\phi _j (\tilde{A} _j)} \frac{1}{|\nabla f |} \, d\sigma - \int_{ \tilde{A} _j} \frac{1}{|\nabla s |} \, d\sigma  \right| = O(h^m) \, ,\qquad 1\leq j \leq d \, .
\end{equation}

This proof is similar to that of item 4 in Lemma \ref{lem:phij2}. For ease of notations, we assume without loss of generality that $j=d$. In this case, $\partial _{\alpha _d} s \neq 0$ on $\tilde{A} _j$,\footnote{As before, this follows for sufficiently small $h$ from the fact that $|\partial _{\alpha _d}f|\geq \frac{\kappa _{\rm f}}{d}$ on $A_d (y)$, and from Theorem~\ref{thm:tensor_spline}.} and so by the implicit function theorem there exists a function~$S$,~such that if $s(\alpha_1, \ldots, \alpha_d)=y$, then $\alpha _d = S(\alpha_1 , \ldots , \alpha _{d-1})$. The domain of $S$ is $$G_{\tilde{A} _d} :\,=  \big\{ (\alpha_1 , \ldots ,\alpha_{d-1} )  ~ | ~ \exists \alpha _d \in [0,1] ~~{\rm s.t.}~~ (\alpha_1, \ldots ,\alpha_d) \in \tilde{A} _d \big\}  \, .$$
In particular, if $(\alpha _1, \ldots ,\alpha_{d-1})\in G_{\tilde{A} _d}$, then $s(\alpha _1, \ldots , \alpha _{d-1}, S(\alpha _1, \ldots ,\alpha _{d-1})) =~y$. Therefore $$\int_{\tilde{A} _d} \frac{1}{|\nabla s|}  \, d\sigma  = \int_{G_{\tilde{A} _d}} \frac{1}{\nabla s\left( \alpha_1, \ldots , \alpha_{d-1} , S(\alpha _1, \ldots , \alpha_{d-1} ) \right)\big|} \sqrt{1+|\nabla S|^2} \, d\alpha _1 \cdots d\alpha _{d-1} \, . $$
Furthermore, by the implicit function theorem, $\partial _{\alpha _j} S = -\frac{\partial _{\alpha _j}s}{\partial _{\alpha _d}s}$ for $1\leq j <d$, and so $$\sqrt{1 + |\nabla S|^2} = \sqrt{1 + \sum _{j=1}^{d-1} \left(\frac{\partial _{\alpha _j}s}{\partial _{\alpha _d}s}\right)^2}  = \frac{1}{\left|\partial _{\alpha _d}s \right|}\sqrt{\left(\partial _{\alpha _d}s\right)^2 + \sum _{j=1}^{d-1} \left(\partial _{\alpha _j}s\right)^2} = \frac{1}{\left|\partial _{\alpha _d}s \right|}|\nabla s| \, .$$ Hence, \begin{subequations}\label{eq:nabla_to_partial}
\begin{equation}
\int_{\tilde{A} _d}\frac{1}{|\nabla s|} \, d\sigma = \int_{G_{\tilde{A} _d}} \frac{1}{\left|\partial _{\alpha _d}s \right|} \,  d\alpha _1 \cdots d\alpha _{d-1} \, .
\end{equation}
Similarly, since $|\partial _{\alpha _d} f|\geq \frac{\kappa _{\rm f}}{2d} > 0$ on $\phi _j (\tilde{A} _j)$, applying the implicit function theorem to $f$ yields a function $F:G_{\phi _d (\tilde{A} _d)} \to \mathbb{R}$ where $G_{\phi _d (\tilde{A} _d)} \subset \mathbb{R}^{d-1}$, such that $f(\alpha _1 , \ldots , \alpha _{d-1}, F(\alpha _1 , \ldots , \alpha _{d-1}))=~y$, and
\begin{equation}
\int_{\phi _d (\tilde{A} _d)}\frac{1}{|\nabla f|} \, d\sigma= \int_{G_{\phi _d (\tilde{A} _d)}} \frac{1}{\left|\partial _{\alpha _d}f\right|} \,  d\alpha' _1 \cdots d\alpha' _{d-1} \, .
\end{equation}
\end{subequations}

Next, by item 2 of Lemma \ref{lem:phij2}, then $\phi _d$ induces a surjective map $\varphi_d: G_{\tilde{A} _d} \to G_{\phi _d (\tilde{A} _d)}$. But, because $\phi _d$ only alters the $\alpha _d$ coordinate, $\varphi_d =  {\rm Id}$, and so $ G_{\tilde{A} _j} = G_{\phi _d (\tilde{A} _d)}$. Substituting this equality and \eqref{eq:nabla_to_partial} into \eqref{eq:Ijy_def} yields
\begin{equation}\label{eq:Iy_premitive_bd}
I_d (y) \leq \left| \int_{G_{\tilde{A} _d}}   \left( \frac{1}{\left|\partial _{\alpha _d}f\right|} - \frac{1}{\left|\partial _{\alpha _d}s\right|}  \right) \,  d\alpha' _1 \cdots d\alpha' _{d-1}\right| \leq \int_{G_{\tilde{A} _d}}  \frac{\left| \partial _{\alpha _d}f - \partial _{\alpha _d}s \right| }{| \partial _{\alpha _d}f|\cdot | \partial _{\alpha _d}s|}  \,  d\alpha' _1 \cdots d\alpha' _{d-1} \, . 
\end{equation}

Bounding \eqref{eq:Iy_premitive_bd} is similar to its one-dimensional counterpart in Appendix \ref{ap:cub_pdf_pf}. Since $|\partial _{\alpha _d} f|>\frac{\kappa _{\rm f}}{d}$, and since by Theorem \ref{thm:tensor_spline}, $\big| \partial _{\alpha _d}s - \partial _{\alpha _d}f \big| \leq C_m h^m$, then for a sufficiently small $h$, $\big| \partial _{\alpha _d}s\big| > \frac{\kappa _{\rm f}}{2d}$ on $\phi _d (\tilde{A} _d)$. Substituting these bounds in~\eqref{eq:Iy_premitive_bd}~yields
\begin{align*}
I_{d}(y) \leq \frac{2d^2}{\kappa _{\rm f}^2}  \int_{G_{\tilde{A} _d}} &|\partial _{\alpha _d} s(\alpha _1, \ldots , \alpha _{d-1}, S(\alpha _1, \ldots, \alpha _{d-1})) \\
&- \partial _{\alpha _d} f(\alpha _1, \ldots , \alpha _{d-1}, F(\alpha _1, \ldots, \alpha _{d-1})) |    d\alpha _1 \cdots d\alpha _{d-1}   \, .
\end{align*}
Next, if we denote $\pmb{\beta} = (\alpha_1, \ldots, \alpha _{d-1}, S(\alpha _1, \ldots ,\alpha _{d-1}))$, then by \eqref{eq:phi_def2}
$$ \phi _d (\pmb{\beta}) = (\alpha_1, \ldots, \alpha _{d-1}, F(\alpha _1, \ldots ,\alpha _{d-1})) \, .$$
Therefore, we can rewrite and bound the left-hand-side integrand by
\begin{equation}\label{eq:integrand_multid_fs_bd}
|\partial _{\alpha _d} s(  \pmb{\beta}) - \partial _{\alpha _d} f(\phi_d (\pmb{\beta}))|\leq |\partial _{\alpha _d} s(\pmb{\beta}) - \partial _{\alpha _d} f(\pmb{\beta})| + |\partial _{\alpha _d} f(\pmb{\beta}) - \partial _{\alpha _d} f(\phi_d (\pmb{\beta}))| \, .
\end{equation}
This bound is very similar to its one-dimensional counterpart in \eqref{eq:l1_bd_middle}. The first term on the right-hand-side of \eqref{eq:integrand_multid_fs_bd} is $O(h^m)$, see Theorem \ref{thm:tensor_spline}. In addition, since $f\in C^2$, the second term in the right-hand-side of \eqref{eq:integrand_multid_fs_bd} reads $$|\partial _{\alpha _d} f(\pmb{\beta}) - \partial _{\alpha _d} f(\phi_d (\pmb{\beta}))| \leq M_2 |\pmb{\beta} -\phi_ d (\pmb{\beta})|  = M_2 |\delta (\pmb{\beta})| = O(h^{m+1}) \, ,$$ 
where, as before, $M_2 = \max _{\Omega} |\partial ^2 _{\alpha _d} f|$ and the last equality is due to \eqref{eq:yalpha_hm}. Applying these bounds to \eqref{eq:integrand_multid_fs_bd} yields
\begin{equation}\label{eq:Iy_bd_fin}
I_{d}(y) \leq \frac{2d^2}{\kappa _{\rm f}^2} \tilde{K}h^m \int_{G_{\tilde{A} _d}}   \,  d\alpha _1 \cdots d\alpha _{d-1}  = K h^m \, ,
\end{equation} 
for some constants $\tilde{K}, K>0$. Moreover, since \eqref{eq:Iy_bd_fin} holds for $I_j (y)$ for all indices $1\leq j \leq d$, then by \eqref{eq:Iy_breakdown} $$I(y) \leq \sum\limits_{j=1}^d I_j (y)+ O(h^m) \leq d Kh^m +O(h^m) \, .$$

\subsection*{Step 6}
{\revdone Although $\| p_f -p_s \|_1 = \int _{-\infty} ^{\infty }I(y) \, dy$, since $\Omega$ is compact and $s$ and $f$ are continuous, $$Q_1 \leq s,(\pmb{\alpha}), f(\pmb{\alpha}) \leq Q_2 \, .$$ and so $I(y) =0$ for $y \not\in [Q_1,Q_2]$. Hence, by \eqref{eq:Iy_bd_fin}
$$\|p_f - p_s\|_q  = \large(\int\limits _{-\infty} ^{\infty} I^q(y) \, dy \large)^{\frac{1}{q}}=  \large( \int\limits_{Q_1}^{Q_2} I^q(y) \, dy \large)^{\frac{1}{q}} \leq \large( K^qh^{qm} (Q_2 - Q_1) \large)^{\frac{1}{q}} \leq K (Q_2 -Q_1)^{\frac{1}{q}} h^m \, .$$ }

\bibliographystyle{siamplain}

\end{document}